\documentclass[a3paper,10pt]{article}
\usepackage{amsmath}
\usepackage{amsfonts}
\usepackage{amssymb}
\usepackage{mathrsfs}
\usepackage{amscd}
\usepackage{pb-diagram}
\usepackage{amsthm}
\usepackage{color}
\usepackage[all]{xy}%pour les diagrams
\usepackage{graphicx}
\usepackage{url}
\author{Mathieu Molitor
\\
\it \small{Department of Mathematics, Keio University}\\
\it \small{3-14-1, Hiyoshi, Kohoku-ku, 223-8522, Yokohama, Japan}\\ 
\small{\it{e-mail:}}\,\,\url{pergame.mathieu@gmail.com}}
\title{Information geometry and the hydrodynamical formulation of quantum mechanics} 
\date{}

\begin{document}
\newtheorem{lemma}{Lemma}[section]
\newtheorem{definition}[lemma]{Definition}
\newtheorem{proposition}[lemma]{Proposition}
\newtheorem{corollary}[lemma]{Corollary}
\newtheorem{theorem}[lemma]{Theorem}
\newtheorem{remark}[lemma]{Remark}
\newtheorem{example}[lemma]{Example}
\bibliographystyle{alpha}

\maketitle 

\begin{abstract}
	Let $(M,g)$ be a compact, connected and oriented Riemannian manifold with volume form $d\textup{vol}_{g}\,.$ We denote 
	$\mathcal{D}$ the space of smooth probability density functions on $M\,,$ i.e. 
	$\mathcal{D}:=\{\rho\in C^{\infty}(M,\mathbb{R})\,\vert\,
	\rho>0\,\,\textup{and}\,\,\int_{M}\rho\cdot d\textup{vol}_{g}=1\}\,.$

	In this paper, we show that the Fr\'{e}chet manifold $\mathcal{D}$ is equipped with a Riemannian metric $g^{\mathcal{D}}$ and 
	an affine connection $\nabla^{\mathcal{D}}$ which are infinite dimensional 
	analogues of the Fisher metric and exponential connection in the context of information geometry. More precisely, we use 
	Dombrowski's construction together with the couple $(g^{\mathcal{D}},\nabla^{\mathcal{D}})$ 
	to get a (non-integrable) almost Hermitian structure on 
	$T\mathcal{D}\,,$ and we show that the corresponding fundamental $2$-form is a symplectic form from which it is possible 
	to recover the usual Schr\"{o}dinger equation for a quantum particle living in $M\,.$ 

	These results echo a recent paper of the author where it is stressed that the Fisher metric and exponential 
	connection are related 
	(via Dombrowski's construction) to K\"{a}hler geometry and quantum mechanics in finite dimension. 
\end{abstract}

\section*{Introduction}
	A \textit{statistical manifold} defined over a measured space $(\Omega,dx)$ 
	is a manifold $S$ together with an injection 
	\begin{eqnarray}
		j\,:\,S\hookrightarrow\bigg\{p\,:\,\Omega\rightarrow\mathbb{R}\,\big\vert\,
		p\,\,\textup{is measurable,}\,\,\,p\geq 0\,\,\,\textup{and}\,\,\,\int_{\Omega}\,
		p(x)\,dx=1\bigg\}\,.
	\end{eqnarray}
	It is known, in the context of \textit{information geometry}
		\footnote{Information geometry 
		is a branch of statistics characterized by its use of differential geometric techniques, 
		see \cite{Amari-Nagaoka,Murray-Rice}.}, 
	that a ``reasonable" statistical manifold $S$
	possesses a uniquely defined \textit{dualistic structure}\footnote{A dualistic structure on a manifold 
	$M$ is a triple $(g,\nabla,\nabla^{*})\,,$ where $g$ is a Riemannian metric and where $\nabla\,,\nabla^{*}$ are affine connections 
	which are dual to each other in the sense that $X\big(g(Y,Z)\big)=g\big(\nabla_{X}Y,Z\big)+g\big(Y,\nabla_{X}^{*}Z\big)$ 
	for all vector fields $X,Y,Z$ on $M\,.$ The connection $\nabla^{*}$ is called the \textit{dual connection} of 
	$\nabla^{*}$ (and vice versa).} $(h_{F},\nabla^{(e)},\nabla^{(m)})\,;$ 
	the metric $h_{F}$ is called the \textit{Fisher metric}, 
	$\nabla^{(e)}$ is the \textit{exponential connection} and $\nabla^{(m)}$ is the \textit{mixture connection}. 
	These geometrical objects encode many important statistical properties of the statistical manifold $S\,.$ For example, they 
	can be used to give lower bounds in estimation problems (e.g. Cram\'{e}r-Rao inequality, see \cite{Amari-Nagaoka,Murray-Rice}). 
	
	The Fisher metric and exponential connection are defined as follows. For a chart $\xi=(\xi_{1},...,\xi_{n})$ 
	of $S\,,$ and denoting $\Gamma_{ij}^{k}$ the Christoffel symbols 
	of $\nabla^{(e)}$ in this chart, we have :
	\begin{description}
	\item[$\bullet$] $(h_{F})_{\xi}\big(\partial_{i},\partial_{j}):=
		E_{p_{\xi}}(\partial_{i}\textup{ln}\,(p_{\xi})\cdot
		\partial_{j}\textup{ln}\,(p_{\xi})\big)\,,$
	\item[$\bullet$] $\Gamma_{ij}^{k}(\xi):=E_{p_{\xi}}\Big[\Big(\partial_{i}\partial_{j}
		\textup{ln}\,(p_{\xi})\cdot\partial_{j}
		\textup{ln}\,(p_{\xi})\Big)\,\partial_{k}\textup{ln}\,(p_{\xi})\Big]\,,$
	\end{description}
	where $E_{p_{\xi}}$ denotes the mean, or expectation, with respect to the probability 
	$p_{\xi}\,dx$ (here $p_{\xi}$ denotes the unique probability density function determined by $\xi$), 
	and where $\partial_{i}$ is a shorthand for $\partial/\partial_{\xi_{i}}\,.$ The connection $\nabla^{(m)}$ 
	is obtained via the duality between $\nabla^{(m)}$ and $\nabla^{(e)}\,.$

	It has recently been stressed in \cite{Molitor-quantique,Molitor-exponential} that dualistic structures on statistical 
	manifolds play a central role in the mathematical foundations of \textit{finite dimensional}\footnote{By finite 
	dimensional, we are referring to quantum systems whose associated Hilbert spaces $\mathcal{H}$ are finite dimensional,
	like for the spin of a particle.} quantum mechanics, in which \textit{Dombrowski's construction} is particularly important. 
	Recall that given a metric $g$ and a connection 
	$\nabla$ on a manifold $M$ ($\nabla$ needs not be the Levi-Civita connection), Dombrowski's construction yields 
	an almost Hermitian structure 
	$(g^{TM},J^{TM},\omega^{TM})$ on $TM\,,$ the latter structure being K\"{a}hler if and only if $\nabla$ and $\nabla^{*}$ are both flat (see
	\cite{Dombrowski,Molitor-exponential}). 
	For example, if $\Omega:=\{x_{1},...,x_{n}\}$ is a finite set and if
	$\mathcal{P}_{n}^{\times}$ is the statistical manifold of nowhere vanishing probabilities $p\,:\,\Omega\rightarrow \mathbb{R}\,,$ $p>0\,,$ 
	$\sum_{k=1}^{n}\,p(x_{k})=1\,,$ then $\nabla^{(e)}$ and $\nabla^{(m)}$ are flat and the K\"{a}hler 
	structure associated to $(h_{F},\nabla^{(e)})$ via Dombrowski's construction 
	on $T\mathcal{P}_{n}^{\times}$ is locally isomorphic to the complex projective space $\mathbb{P}(\mathbb{C}^{n})$ 
	of complex lines in $\mathbb{C}^{n}$ \cite{Molitor-quantique}.

	This example, although mathematically simple, is physically fundamental. For, as
	it is known, a finite dimensional quantum system can be entirely described by 
	the K\"{a}hler structure of the complex projective space 
	$\mathbb{P}(\mathbb{C}^{n})$ associated to the Hilbert space $\mathbb{C}^{n}$ of quantum states\footnote{This is also 
	true for infinite dimensional quantum systems.}; this is the 
	so-called \textit{geometrical formulation} of quantum mechanics (see for example \cite{Ashtekar}). 
	Hence, by realizing an open dense subset of $\mathbb{P}(\mathbb{C}^{n})$ as a 
	``K\"{a}hlerification" of $\mathcal{P}_{n}^{\times}$ via Dombrowski's construction (see \cite{Molitor-exponential} 
	for a precise statement), we directly 
	connect information geometry to quantum mechanics.

	Based on these observations, and together with other mathematical results, we were led in \cite{Molitor-exponential} to conjecture 
	that the quantum formalism, at least in finite dimension, has a purely information-theoretical origin in which the Fisher metric and 
	the exponential connection, together with Dombrowki's construction, are crucial.\\

	The purpose of the present paper is to investigate the properties of a particular \textit{infinite dimensional} quantum system in 
	the light of the results obtained in \cite{Molitor-quantique,Molitor-exponential}. Our quantum system is a non-relativistic 
	quantum particle, mathematically represented by a wave function 
	$\psi\,:\,M\rightarrow \mathbb{C}\,,$ living on a compact and connected Riemannian manifold 
	$(M,g)\,,$ and whose dynamics is governed by the Schr\"{o}dinger equation 
	\begin{eqnarray}\label{equation Schrodinger encore et oui!}
		i\hslash \,\dfrac{\partial \psi}{\partial t}=-\dfrac{\hslash^{2}}{2}\,\Delta\,\psi+V\psi\,,
	\end{eqnarray}
	where $\hbar$ is Planck constant, $\Delta$ is the Laplacian operator and where 
	$V\,:\,M\rightarrow \mathbb{R}$ is a given potential.
		
	To this system, we attach, as a statistical model, the space $\mathcal{D}$ of smooth 
	density probability functions on $M\,:$ 
	\begin{eqnarray}
		\mathcal{D}:=\big\{\rho\in C^{\infty}(M,\mathbb{R})\,\big\vert\,\rho>0\,,\,\,\,
		\smallint_{M}\,\rho\,d\textup{vol}_{g}=1\big\}\,,
	\end{eqnarray}
	where $d\textup{vol}_{g}$ denotes the Riemannian volume form associated to $g$
	($M$ is assumed oriented). We regard the space $\mathcal{D}$ as an infinite dimensional 
	analogue of $\mathcal{P}_{n}^{\times}\,.$ 

	In this paper, we show the following: first, that it is possible 
	to rewrite the Schr\"{o}dinger equation \eqref{equation Schrodinger encore et oui!} into a genuine 
	system of Lagrangian equations on $T\mathcal{D}$ for an appropriate Lagrangian 
	$\mathcal{L}\,:\,T\mathcal{D}\rightarrow \mathbb{R}\,.$ Second, that this Lagrangian system 
	can be reformulated in a symplectic way on $T\mathcal{D}$ using geometric mechanical methods. 
	Finally, that the corresponding symplectic form $\Omega_{\mathcal{L}}$ on $T\mathcal{D}$ 
	is nothing but the fundamental form of the almost Hermitian structure associated, via Dombrowski's construction, to 
	a natural metric $g^{\mathcal{D}}$ and a connection $\nabla^{\mathcal{D}}$ living on $\mathcal{D}\,.$  
	
	The couple $(g^{\mathcal{D}},\nabla^{\mathcal{D}})$ on $\mathcal{D}$
	is thus --and this is the main observation of this paper-- 
	an infinite dimensional analogue of $(h_{F},\nabla^{(e)})$ on $\mathcal{P}_{n}^{\times}$ 
	which encodes the dynamics of the quantum particle, exactly as in the finite dimensional case (see \cite{Molitor-exponential}). 

	Additionally, we observe that the almost complex structure of $T\mathcal{D}$ is not integrable and that, contrary to
	$\nabla^{(e)}\,,$ the connection $\nabla^{\mathcal{D}}$ on $\mathcal{D}$ has a non-trivial torsion (this proves in particular that 
	$\nabla^{\mathcal{D}}$ it \textit{not} 
	the Levi-Civita connection associated to $g^{\mathcal{D}}$).\\

	This paper is organized as follows. In \S\ref{subsection tangent} we describe the 
	geometry of $\mathcal{D}$ and its tangent bundle; that will allow us, 
	in \S\ref{section Euler-Lagrange etc.} and 
	\S\ref{section Hamiltonian}, to recast the Schr\"{o}dinger equation directly on $T\mathcal{D}\,,$ in a 
	Lagrangian form (\S\ref{section Euler-Lagrange etc.}) and in a Hamiltonian form 
	(\S\ref{section Hamiltonian}). Finally, in \S\ref{sss avant dernier} we observe that the symplectic form 
	$\Omega_{\mathcal{L}}$ on $T\mathcal{D}$ describing  the dynamics of the quantum particle 
	is nothing but the fundamental form of the 
	almost Hermitian structure associated to $(g^{\mathcal{D}},\nabla^{\mathcal{D}})$ on $\mathcal{D}\,.$	
	The paper ends with \S\ref{sss la fin} where we discuss a possible definition for the wave function 
	associated to a moving probability. An example is considered. 

	Some of our results are expressed in the category of \textit{tame Fr\'{e}chet manifolds} introduced 
	by Hamilton in \cite{Hamilton}. The relevant definitions are recalled in \S\ref{subsection the category of hamitlon}. 

\section{Hamilton's category of tame Fr\'echet manifolds}
\label{subsection the category of hamitlon}
	In this section, we review very briefly the category of tame Fr\'{e}chet manifolds 
	introduced by Hamilton in \cite{Hamilton}. 
\begin{definition}
		\begin{enumerate}
			\item A graded Fr\'echet space $(F,\{\|\,.\,\|_{n}\}_{n\in\mathbb{N}})\,,$
				is a Fr\'echet space $F$ whose topology is defined by a collection of seminorms 
				$\{\|\,.\,\|_{n}\}_{n\in\mathbb{N}}$ which are increasing in strength:
				\begin{eqnarray}
					\|x\|_{0}\leq\|x\|_{1}\leq\|x\|_{2}\leq \cdots
				\end{eqnarray}
				for all $x\in F\,.$
			\item A linear map $L\,:\,F\rightarrow G$ between two graded Fr\'echet spaces
				$F$ and $G$ is tame (of degree $r$ and base $b$) if for all $n\geq b\,,$
				there exists a constant $C_{n}>0$ such that for all $x\in F\,,$
				\begin{eqnarray}
					\|L(x)\|_{n}\leq C_{n}\,\|x\|_{n+r}\,.
				\end{eqnarray}
			\item If $(B,\|\,.\,\|_{B})$ is a Banach space, then $\Sigma(B)$ denotes 
				the graded Fr\'echet space of all sequences $\{x_{k}\}_{k\in\mathbb{N}}$ of 
				$B$ such that for all $n\geq 0,$ 
				\begin{eqnarray}
					\|\{x_{k}\}_{k\in\mathbb{N}}\|_{n}:=\displaystyle\Sigma_{k=0}^{\infty}\,
					e^{nk}\|x_{k}\|_{B}<\infty\,.
				\end{eqnarray}
			\item A graded Fr\'echet space $F$ is tame if there exist a Banach space 
				$B$ and two tame linear maps $i\,:\,F\rightarrow \Sigma(B)$ and 
				$p\,:\,\Sigma(B)\rightarrow F$ such that $p\circ i$ is the identity on $F\,.$
			\item Let $F,G$ be two tame Fr\'echet spaces, $U$ an open subset of 
				$F$ and $f\,:\,U\rightarrow G$ a map. We say that $f$ is a smooth tame map
				if $f$ is smooth\footnote{By smooth we mean that 
				$f\,:\,U\subseteq F\rightarrow G$ 
				is continuous and that 
				for all $k\in\mathbb{N}\,,$ the $k$th derivative 
				$d^{k}f\,:\,U\times F\times \cdots \times F
				\rightarrow G$ exists and is jointly continuous on the product  space, such as 
				described in \cite{Hamilton}.\label{footnote}} 
				and if for every $k\in\mathbb{N}$ and for every 
				$(x,u_{1},...,u_{k})\in U\times F\times \cdots F\,,$ there exist a neighborhood
				$V$ of $(x,u_{1},...,u_{k})$ in $U\times F\times \cdots F$ and 
				$b_{k},r_{0},...,r_{k}\in\mathbb{N}$ such that for every $n\geq b_{k}\,,$ there 
				exists $C_{k,n}^{V}>0$ such that 
				\begin{eqnarray}
					&&\|d^{k}f(y)\{v_{1},...,v_{k}\}\|_{n}
					\,\,\leq \,\,C_{k,n}^{V}\,\big(1+\|y\|_{n+r_{0}}
					+\|v_{1}\|_{n+r_{1}}+\cdots+\|v_{k}\|_{n+r_{k}}\big)\,,
				\end{eqnarray}
				for every $(y,v_{1},...,v_{k})\in V\,,$ where $d^{k}f\,:\,
				U\times F\times\cdots\times F\rightarrow G$ denotes the $k$th derivative of 
				$f\,.$
		\end{enumerate}
\end{definition}
\begin{remark}
	In this paper, we use interchangeably the notation 
	$(df)(x)\{v\}$ or $f_{*_{x}}v$ for the first derivative of $f$ at a 
	point $x$ in direction $v\,.$ 
\end{remark}
	As one may notice, tame Fr\'echet spaces and smooth tame maps form 
	a category, and it is thus natural to define a tame Fr\'echet manifold 
	as a Hausdorff topological space with an atlas of coordinates charts taking their value in
	tame Fr\'echet spaces, such that the coordinate transition functions are all
	smooth tame maps (see \cite{Hamilton}).
	The definition of a tame smooth map between tame Fr\'echet manifolds is then 
	straightforward, and we thus obtain a subcategory of the category of Fr\'echet manifolds.\\
	In order to avoid confusion, let us also make precise our notion of submanifold. We will say 
	that a subset $\mathcal{M}$ of a 
	tame Fr\'echet manifold $\mathcal{M}\,,$ endowed with the trace topology, 
	is a submanifold, 
	if for every point $x\in \mathcal{M}\,,$ there exists a chart 
	$(\mathcal{U},\varphi)$ of $\mathcal{M}$ such that $x\in \mathcal{U}$ and such that 
	$\varphi(\mathcal{U}\cap \mathcal{M})=U\times \{0\}\,,$ where 
	$\varphi(\mathcal{U})=U\times V$ is a product of two open subsets of tame Fr\'echet spaces. 
	Note that a submanifold of a tame Fr\'echet manifold is also a tame Fr\'echet manifold.\\

	For the sake of completeness, let us state here the raison d'\^{etre} of 
	tame Fr\'echet spaces and tame Fr\'echet manifolds (see \cite{Hamilton}) :
\begin{theorem}[Nash-Moser inverse function Theorem]
	Let $F,G$ be two tame Fr\'echet spaces, $U$ an open subset of 
	$F$ and $f\,:\, U\rightarrow G$ a smooth tame map. If there exists 
	an open subset $V\subseteq U$ such that 
	\begin{enumerate}
		\item $df(x)\,:\,F\rightarrow G$ is an linear isomorphism for all 
			$x\in V\,,$
		\item the map $V\times G\rightarrow F,\,(x,v)\mapsto 
			\big(df(x)\big)^{-1}\{v\}$ is a smooth tame map, 
	\end{enumerate}
	then $f$ is locally invertible on $V$ and each local inverse is a
	smooth tame map. 
\end{theorem}
\begin{remark}
	The Nash-Moser inverse function Theorem is important in geometric hydrodynamics, 
	since one of its most important geometric objects, namely the group 
	of all smooth volume preserving diffeomorphims 
	$\textup{SDiff}_{\mu}(M):=\{\varphi\in\textup{Diff}(M)\,\vert\,\varphi^{*}\mu=\mu\}$ 
	of an oriented manifold $(M,\mu)\,,$ 
	can only be given a rigorous Fr\'echet Lie group structure by
	using an inverse function theorem (at least up to now).
	To our knowledge, only two authors succeeded in doing this. The first was Omori 
	who showed and used an inverse function theorem in terms of 
	ILB-spaces (``inverse limit of Banach spaces", see \cite{Omori}), and later on, Hamilton with 
	his category of tame Fr\'echet spaces together with the Nash-Moser inverse function 
	Theorem (see \cite{Hamilton}). Nowadays, it is nevertheless not uncommon 
	to find  mistakes or 
	big gaps in the literature when it comes to the differentiable 
	structure of $\textup{SDiff}_{\mu}(M)\,,$
	even in some specialized textbooks in infinite dimensional geometry. The case 
	of $M$ being non-compact is even worse, and of course, no proof that $\textup{SDiff}_{\mu}(M)$ 
	is a ``Lie group" is available in this case.
\end{remark}
	Finally, and quite apart from the category of Hamilton, let us remind one of the most 
	useful result of the convenient calculus 
	(see \cite{Kriegl-Michor}) :
\begin{lemma}\label{lemme convenient calculus encore}
	Let $F,G$ be two Fr\'echet spaces, $U$ an open subset of $F$ and $f\,:\,U\rightarrow G$
	a map. Then $f$ is smooth is the sense of Hamilton (see footnote \ref{footnote}), 
	if and only if 
	$f\circ c\,:\,I\rightarrow \mathbb{R}$ is a smooth curve in $G$ whenever
	$c\,:\,I\rightarrow U$ is a smooth curve in $U\,.$
\end{lemma}
	As one may show, if $M,N$ are manifolds, $M$ being compact, then 
	a smooth curve in the Fr\'echet manifold 
	$C^{\infty}(M,N)$ may by identified with a smooth map $f\,:\,I\times M\rightarrow N\,,$
	its time derivative being identified with the partial derivative of $f$ 
	with respect to $t\,.$ From this together with Lemma \ref{lemme convenient calculus encore}\,,
	it is usually easy to show that a map defined between 
	submanifolds of spaces of maps is smooth : it suffices to compose this map 
	with a smooth curve, and then to check that the result is smooth in the ``finite 
	dimensional sense" with respect to all the ``finite dimensional" variables (see 
	\cite{Kriegl-Michor}).  

\section{The manifold structure of $\mathcal{D}$ and its tangent bundle}\label{subsection tangent}
\label{section manifold structure...}
	Let $(M,g)$ be a compact, connected and oriented Riemannian manifold with Riemannian volume form $d\textup{vol}_{g}\,,$ and let 
	$\mathcal{D}$ be the space of smooth density probability functions on $M\,:$ 
	\begin{eqnarray}
		\mathcal{D}:=\big\{\rho\in C^{\infty}(M,\mathbb{R})\,\big\vert\,\rho>0\,,\,\,\,
		\smallint_{M}\,\rho\,d\textup{vol}_{g}=1\big\}\,.
	\end{eqnarray}
	Throughout this section, we shall write $C^{\infty}(M)$ instead of $C^{\infty}(M,\mathbb{R})$ (and 
	similar for subspaces of $C^{\infty}(M,\mathbb{R})$) if there is no danger of confusion. We shall also 
	use the notation $\mathbb{R}_{+}^{*}:=\{r\in\mathbb{R}\,\vert\,r>0\}\,.$\\

	Let us start with the differentiable structure of $\mathcal{D}\,.$
\begin{proposition}\label{Proposition D manifold}
	The space $\mathcal{D}$ is a tame Fr\'echet submanifold of the tame Fr\'echet space 
	$C^{\infty}(M)\,,$ and for $\rho\in\mathcal{D}\,,$
	\begin{eqnarray}\label{equation 1 identification tangent}
		T_{\rho}\mathcal{D}\cong C_{0}^{\infty}(M)\,,
	\end{eqnarray}
	where 
	\begin{eqnarray}
		C_{0}^{\infty}(M):=\big\{f\in C^{\infty}(M)\,\big\vert\,
		\smallint_{M}\,f\,d\textup{vol}_{g}=0\big\}\,.
	\end{eqnarray}
\end{proposition}
	Observe that we have the following $L^{2}$-orthogonal decomposition,
	\begin{eqnarray}\label{equation decomposition L2}
		C^{\infty}(M)=C_{0}^{\infty}(M)\oplus \mathbb{R}\,,
	\end{eqnarray}
	the decomposition being given, for $f\in C^{\infty}(M)\,,$ by 
	\begin{eqnarray}
		f=f-\dfrac{1}{\textup{Vol}(M)}\int_{M}\,f\,\cdot d\textup{vol}_{g}
		+\dfrac{1}{\textup{Vol}(M)}\int_{M}\,f\,\cdot d\textup{vol}_{g}\,,
	\end{eqnarray}
	where $\textup{Vol}(M):=\smallint_{M}\,d\textup{vol}_{g}$ denotes the Riemannian volume 
	of $M\,.$ In particular, the space $C_{0}^{\infty}(M)$ is a tame Fr\'echet space
	(it is a Fr\'echet space because $C_{0}^{\infty}(M)$ is closed in 
	$C^{\infty}(M)$ and it is also a 
	tame space because $C^{\infty}(M)$ is tame, see \cite{Hamilton}, 
	Definition 1.3.1 and Corollary 1.3.9).
\begin{proof}[Proof of Proposition \ref{Proposition D manifold}]
	The proof relies on the following tame diffeomorphim of tame F\'echet manifolds:
	\begin{eqnarray}
		\Phi\,:\,
		\left \lbrace
			\begin{array}{cc}
				C^{\infty}(M)\rightarrow C_{0}^{\infty}(M)\times
				\mathbb{R}\,,\\
				f\mapsto \Big(f-\textup{Vol}(M)^{-1}\int_{M}\,f\,d\textup{vol}_{g},
				\textup{Vol}(M)^{-1}\int_{M}\,f\,d\textup{vol}_{g}-\textup{Vol}(M)^{-1}\Big)\,.
			\end{array}
		\right.
	\end{eqnarray}
	Using $\Phi\,,$ it is possible to define splitting charts for $C^{\infty}(M)\,;$
	indeed, the space $C^{\infty}(M,\mathbb{R}_{+}^{*})$ being clearly an open subset of 
	$C^{\infty}(M)$ for its natural Fr\'echet 
	space topology, every $\rho\in \mathcal{D}$ possesses an open 
	neighborhood in $C^{\infty}(M)\,,$ say $U_{\rho}\,,$ 
	such that $\rho\in U_{\rho}\subseteq C^{\infty}(M,\mathbb{R}_{+}^{*})\,,$ and, 
	restricting $U_{\rho}$ if necessary, we may assume that 
	$\Phi(U_{\rho})=V_{\rho}\times W_{\rho}$ where $V_{\rho}$ and $W_{\rho}$ are open subsets of 
	$C_{0}^{\infty}(M)$ and $\mathbb{R}$ respectively. But now, 
	$(U_{\rho},\Phi\vert_{U_{\rho}})$ is a chart of $C^{\infty}(M)$ and it is 
	easy to see that 
	\begin{eqnarray}
		(\Phi\vert_{U_{\rho}})(U_{\rho}\cap \mathcal{D})=V_{\rho}\times \{0\}\,.
	\end{eqnarray}
	The proposition follows. 
\end{proof}
	We now want to give a geometrical description of the tangent 
	space of $\mathcal{D}\,.$ Recall that if $X\in \mathfrak{X}(M)$ is a vector field 
	on $M\,,$ then its divergence with respect to the volume form 
	$d\textup{vol}_{g}$ is the unique function 
	$\textup{div}(X)\,:\,M\rightarrow\mathbb{R}$ satisfying 
	$\mathscr{L}_{X}(d\textup{vol}_{g})=\textup{div}(X)\cdot d\textup{vol}_{g}\,,$
	$\mathscr{L}_{X}$ being the Lie derivative in direction $X\,.$\\
	Using the divergence operator, we define, for $f\,:\,M\rightarrow \mathbb{R}_{+}^{*}\,,$ 
	an elliptic differential operator $\textup{P}_{f}\,:\,
	C^{\infty}(M)\rightarrow C^{\infty}(M)$ 
	via the formula
	\begin{eqnarray}
		\textup{P}_{f}(u):=\textup{div}(f\cdot\nabla u)\,,
	\end{eqnarray}
	where $u\,:\,M\rightarrow \mathbb{R}$ is a smooth function. Observe that 
	\begin{description}
		\item[$\bullet$] $\textup{P}_{1}=\Delta$ is the Laplacian operator\,,
		\item[$\bullet$] $\textup{P}_{f}$ takes values in 
			$C_{0}^{\infty}(M)$ since the integral with respect to the 
			Riemannian volume form of a divergence 
			is always zero by application of Stokes' Theorem.  
		\item[$\bullet$] The kernel of $\textup{P}_{f}$ reduces to the constant functions. 
			This is due to the fact that $\textup{P}_{f}$ is a second order elliptic differential 
			operator whose constant term $\textup{P}_{f}(1)$ is zero, and 
			it is well known that for such differential operators on compact manifolds, 
			the kernel reduces to the constant functions (see \cite{Jost}).
	\end{description}
\begin{lemma}\label{lemme operateur inversible}
	For $f\in C^{\infty}(M)\,,$ $f>0\,,$ the restriction $\overline{\textup{P}}_{f}$
	of the operator $\textup{P}_{f}$ to $C_{0}^{\infty}(M)\,,$ 
	\begin{eqnarray}
		\overline{\textup{P}}_{f}\,:\,C_{0}^{\infty}(M)
		\rightarrow C_{0}^{\infty}(M)\,,
	\end{eqnarray}
	is an isomorphism of Fr\'echet spaces. Moreover, its family of inverses
	\begin{eqnarray}\label{equation inverse famille operateur elliptiques²}
		C^{\infty}(M,\mathbb{R}_{+}^{*})\times C_{0}^{\infty}(M)\rightarrow 
		C_{0}^{\infty}(M),\,\,(f,h)\mapsto (\overline{\textup{P}}_{f})^{-1}(h)
	\end{eqnarray}
	forms a smooth tame map. 
%	where $\mathbb{R}_{+}^{*}:=\{r\in\mathbb{R}\,\vert\,r>0\}\,.$
\end{lemma}
\begin{proof}
	The operator $\overline{\textup{P}}_{f}$ is injective since its kernel is the intersection 
	of the kernel of ${\textup{P}}_{f}$ with the space $C_{0}^{\infty}(M)\,,$
	which is zero.		
	
	For the surjectivity, take $\tilde{f}\,:\,[0,1]\rightarrow C^{\infty}(M,\mathbb{R}_{+}^{*})$ 
	a continuous path such that $\tilde{f}_{0}\equiv 1$ and 
	$\tilde{f}_{1}=f\,.$ As one may see, 
	$\textup{P}_{\tilde{f}_{t}}$ defines a continuous path of elliptic 
	operators (acting on a suitable Sobolev space), and by the topological invariance 
	of the analytic index $\textup{Ind}$ of an elliptic operator together with the fact
	that the analytic index of $\Delta\,:\,C^{\infty}(M)\rightarrow 
	C^{\infty}(M)$ is zero, we have :
	\begin{eqnarray}
		\textup{Ind}({\textup{P}}_{f})=\textup{Ind}({\textup{P}}_{\tilde{f}_{1}})
		=\textup{Ind}({\textup{P}}_{\tilde{f}_{0}})=\textup{Ind}(\Delta)=0\,. 
	\end{eqnarray}
	Hence, the codimension of the image $\textup{Im}(\textup{P}_{f})$ of 
	$\textup{P}_{f}$ is 1, and since  $\textup{Im}(\textup{P}_{f})\subseteq 
	C_{0}^{\infty}(M)\,,$ this later space being of codimension 1, 
	$\textup{Im}(\textup{P}_{f})=C_{0}^{\infty}(M)\,.$ 
	It follows that $\overline{\textup{P}}_{f}\,:\,C_{0}^{\infty}(M)\rightarrow
	C_{0}^{\infty}(M)$ is a bijection.

	Finally, $\overline{\textup{P}}_{f}$ is continuous since it is a differential operator, 
	and its inverse is also continuous by application of the open mapping Theorem.

	The fact that the family of inverses defined in 
	\eqref{equation inverse famille operateur elliptiques²} forms a smooth tame map 
	is a consequence of a result due to Hamilton (see \cite{Hamilton}, Theorem 3.3.3) 
	about the family of inverses of a family of invertible (up to something of finite dimension)
	elliptic differential operators, applied to the following map :
	\begin{eqnarray}
		\left \lbrace
			\begin{array}{cc}
				C^{\infty}(M,\mathbb{R}_{+}^{*})\times C^{\infty}(M)\times \mathbb{R}
				\rightarrow C^{\infty}(M)\times \mathbb{R}\,,\\
				(f,h,x)\mapsto 
				\big(\textup{P}_{f}(h)+x,\smallint_{M}\,h\,d\textup{vol}_{g}\big)\,.
			\end{array}
		\right.
	\end{eqnarray}
	The result of Hamilton implies the existence of a smooth Green operator 
	$\textup{G}\,:\,C^{\infty}(M,\mathbb{R}_{+}^{*})\times C^{\infty}(M)\rightarrow
	C_{0}^{\infty}(M)$ whose restriction to $C^{\infty}(M,\mathbb{R}_{+}^{*})\times 
	C_{0}^{\infty}(M)$ coincides with the family considered in 
	\eqref{equation inverse famille operateur elliptiques²}\,.
	The lemma follows. 
\end{proof}
	%%%%%%%%%%%%%%%%%%%%%%%%
	%	In view of Lemma \ref{lemme operateur inversible}\,, we see that for $\rho\in \mathcal{D}$
	%	and $h\in T_{\rho}\mathcal{D}\cong C^{\infty}_{0}(N,\mathbb{R})\,,$ there exists a unique 
	%	function $\phi\,:\,N\rightarrow \mathbb{R}$ (defined up to an additive constant), such that 
	%	\begin{eqnarray}\label{equation discussion avec dorian}
	%		h=\textup{div}(\rho\cdot \nabla\phi)\,.
	%	\end{eqnarray}
	%	The tangent space $T_{\rho}\mathcal{D}$ may thus be identified 
	%	with the space of all gradients 
	%	$\nabla \phi\,:$
	%	\begin{eqnarray}\label{equation identification flotte}
	%		T_{\rho}\mathcal{D}\cong \big\{\nabla \phi\in\mathfrak{X}(N)\,\vert\, 
	%		\phi\in C^{\infty}(N,\mathbb{R})\big\}\,.
	%	\end{eqnarray}
	%	An identification such as \eqref{equation identification flotte} 
	%	of $T_{\rho}\mathcal{D}$ is more ``hydrodynamical" in nature than
	%	the one given in \eqref{equation 1 identification tangent}; this offers the possibility, 
	%	having in mind geometric fluid mechanics, to find more easily natural Lagrangians 
	%	on $T\mathcal{D}\,.$ \\
	%	
	%	More generally, Lemma \ref{lemme operateur inversible} may be used to find 
	%	many hydrodynamical descriptions of 
	%	$T\mathcal{D}\,,$ and, since we will also consider electromagnetism in the sequel, 
	%	we will actually use the following slight generalization 
	%	of \eqref{equation discussion avec dorian} and 
	%	\eqref{equation identification flotte}\,, that we formalize as follows : 
	%%%%%%%%%%%%%%%%%%%%%%%%%%%%%%%%%%%
\begin{proposition}\label{proposition identification gradient}
	Let $X\in \mathfrak{X}(M)$ be a vector field and let $\rho\in \mathcal{D}$ be a smooth 
	density. For $h\in T_{\rho}\mathcal{D}\cong C_{0}^{\infty}(M)\,,$ there exists
	a unique function $\phi\,:\,M\rightarrow\mathbb{R}$ (defined up to an additive constant), such 
	that 
	\begin{eqnarray}
		h=\textup{div}\,\big(\rho\,(\nabla\phi+X)\big)\,.
	\end{eqnarray}
	Moreover, the map 
	\begin{eqnarray}\label{equation definition de l'identification}
		T\mathcal{D}\rightarrow \mathcal{D}\times \nabla C^{\infty}(M),\,\,
		h=\textup{div}\,\big(\rho\,(\nabla\phi+X)\big)\mapsto (\rho,\nabla\phi)\,,
	\end{eqnarray}
	is a non-linear tame isomorphism of tame Fr\'echet vector bundles, 
	$\mathcal{D}\times \nabla C^{\infty}(M)$ being the trivial vector bundle over
	$\mathcal{D}\,.$
\end{proposition}
\begin{proof}
	For $\rho\in \mathcal{D}$ and $h\in T_{\rho}\mathcal{D}\cong 
	C_{0}^{\infty}(M)\,,$ 
	define $\phi\in C_{0}^{\infty}(M)$ by letting 
	\begin{eqnarray}\label{equation j'ai fain!! et oui!}
		\phi:=(\overline{P}_{\rho})^{-1}\big[\,h-\textup{div}(\rho\,X)\big]
	\end{eqnarray}
	(note that $\textup{div}(\rho\,X)\in C_{0}^{\infty}(M)\,,$ and thus 
	$h-\textup{div}(\rho\,X)\in C_{0}^{\infty}(M)$). \\
	By applying the operator $\overline{P}_{\rho}$ to \eqref{equation j'ai fain!! et oui!}, 
	we see that 
	\begin{eqnarray}
		\overline{P}_{\rho}(\phi)=h-\textup{div}(\rho\,X)\,\,\,&\Rightarrow& \,\,\,
		\textup{div}(\rho\,\nabla\phi)=h-\textup{div}(\rho\,X)\nonumber\\
		&\Rightarrow& \,\,\,h=\textup{div}\,\big(\rho\,(\nabla\phi+X)\big)\,.
	\end{eqnarray}
	Moreover, if $\phi'\,:\,M\rightarrow \mathbb{R}$ satisfies 
	$h=\textup{div}\,\big(\rho\,(\nabla\phi'+X)\,,$ then $P_{\rho}(\phi-\phi')=0\,,$ 
	and thus, $\phi-\phi'$ is a constant function. 
	The first assertion of the proposition follows.\\
	For the second assertion, it is clear that the map defined in 
	\eqref{equation definition de l'identification} is a fiber preserving bijection; its
	smoothness is a consequence of the smoothness of the family of inverses 
	\eqref{equation inverse famille operateur elliptiques²} (that one may apply in charts 
	such as defined in the proof of Proposition \ref{Proposition D manifold}, 
	or directly using the convenient calculus developed in \cite{Kriegl-Michor}); this map 
	is also tame for the same reason and its 
	inverse is clearly a smooth tame map. The proposition follows. 
\end{proof}
\begin{remark}
	The space of all gradients $\nabla C^{\infty}(M)$ is a tame Fr\'echet 
	space. This comes from the fact that the Helmholtz-Hodge decomposition 
	\begin{eqnarray}\label{equation HH}
		\mathfrak{X}(M)=
		\mathfrak{X}_{d\textup{vol}_{g}}(M)\oplus \nabla C^{\infty}(M)\,,
	\end{eqnarray}
	where 
	$\mathfrak{X}_{d\textup{vol}_{g}}(M):=\{X\in \mathfrak{X}(M)\,\vert\,\textup{div}(X)=0\}\,,$
	is a topological direct sum (see \cite{Hamilton}). As a consequence,
	the space $\mathcal{D}\times \nabla C^{\infty}(M)$ is a tame Fr\'echet space, 
	and in particular, it is a trivial tame Fr\'echet vector bundle over $\mathcal{D}\,.$
\end{remark}
\begin{remark}\label{remarque identification avec t}
	In connection with electromagnetism, if we allow the vector field $X\in \mathfrak{X}(M)$ 
	of Proposition \ref{proposition identification gradient} to be time-dependent, 
	then an obvious modification of the proof of Proposition 
	\ref{proposition identification gradient} shows that the map
	\begin{eqnarray}
		\left \lbrace
			\begin{array}{cc}
				T\mathcal{D}\times \mathbb{R}\rightarrow \mathcal{D}\times \nabla 
				C^{\infty}(M)\times \mathbb{R}\,,\nonumber\\
				\Big(\rho,h=\textup{div}\,\big(\rho\,(\nabla\phi_{t}+X_{t})\big),t\Big)\mapsto 
				(\rho,\nabla\phi_{t},t)\,,
			\end{array}
		\right.
	\end{eqnarray}
	is a smooth tame diffeomorphism. 
\end{remark}
	
\begin{remark}
	In \S\ref{subsection the category of hamitlon} and \S\ref{section manifold structure...}
	we were working in the category of tame Fr\'echet spaces, but in the sequel 
	we will relax this hypothesis and simply work with the usual Fr\'echet category. 
\end{remark}

\section{Euler-Lagrange equations on $\mathcal{D}$ and the Schr\"{o}dinger equation}
	\label{section Euler-Lagrange etc.}

	 Having a precise and geometric description of the tangent bundle of $\mathcal{D}\,,$
	it is easy to write interesting Lagrangians on $\mathcal{D}\,.$ Indeed, for a time-dependent 
	vector field $X_{t}\in\mathfrak{X}(M)\,,$ a time-dependent potential $V_{t}\,:\,M\rightarrow 
	\mathbb{R}\,,$ and using the diffeomorphism 
	$T\mathcal{D}\times \mathbb{R}\rightarrow \mathcal{D}\times \nabla 
	C^{\infty}(M)\times \mathbb{R}$ of Remark \ref{remarque identification avec t},
	we can consider, with an abuse of notation, the following time-dependent Lagrangian :
	\begin{eqnarray}\label{equation definition bis du Lagrangien}
		&&{\mathcal{L}}\Big(\rho,h=\textup{div}\,\big(\rho\,(\nabla\phi_{t}+X_{t})\big),t\Big)=
			{\mathcal{L}}_{}(\rho,\nabla\phi_{t},t):=\nonumber\\
		&&\int_{M}\,\Big(\dfrac{1}{2}\,\|\nabla\phi_{t}\|^{2}-\|X_{t}\|^{2}-V_{t}
			\Big)\,\rho\cdot d\textup{vol}_{g}
			-\dfrac{\hslash^{2}}{2}\,\int_{M}\,\big\|\nabla(\sqrt{\rho}\,)\big\|^{2}\cdot 
			d\textup{vol}_{g}\,.\,\,\,\,\,\,\,\,\text{}
	\end{eqnarray}
	Note that ${\mathcal{L}}$ is smooth by application of the convenient calculus together 
	with Remark \ref{remarque identification avec t}.\\

	By using the formula 
	\begin{eqnarray}
		\dfrac{1}{4}\,\dfrac{\|\nabla u\|^{2}}{u^{2}}
			-\dfrac{1}{2}\,\dfrac{\Delta u}{u}=-\dfrac{\Delta\big(\sqrt{u}\big)}{\sqrt{u}}\,,
	\end{eqnarray}
	which is valid for every smooth function $u\,:\,M\rightarrow \mathbb{R}\,,$ and by doing 
	an usual fixed end-point variation of the Lagrangian ${\mathcal{L}}\,,$ one easily 
	finds the following Euler-Lagrange equations :

\begin{proposition}\label{proposition les Euler-Lagrange}
	The Euler-Lagrange equations associated to the Lagrangian ${\mathcal{L}}_{}$ defined 
	in \eqref{equation definition bis du Lagrangien}, are given by 
	\begin{eqnarray}\label{equation les equations d'euler-lagrange, enfin!!!}	
		\left\lbrace
			\begin{array}{ccc}
				\dfrac{\partial \phi}{\partial t}&=&\dfrac{1}{2}\,\|\nabla\phi+X\|^{2}+V
				-\dfrac{\hslash^{2}}{2}\,\dfrac{\Delta\,\big(\sqrt{\rho}\,\big)}
				{\sqrt{\rho}}+c_{t}\,,\\
				\dfrac{\partial \rho}{\partial t}&=&\textup{div}\,
				\big(\rho\,(\nabla\phi+X)\big)\,,
			\end{array}
		\right.
	\end{eqnarray}
	where $\rho\,:\,I\subseteq \mathbb{R}\rightarrow \mathcal{D}$ 
	is a smooth curve in $\mathcal{D}$ and where 
	$c_{t}$ is a time-dependent constant. 
\end{proposition}
\begin{remark}
	The second equation in \eqref{equation les equations d'euler-lagrange, enfin!!!} has 
	actually nothing to do with variational principles; its is only the geometric way 
	to express tangent vectors in $\mathcal{D}\,,$ such as described in Proposition 
	\ref{proposition identification gradient}.
\end{remark}
\begin{remark}
	Due to similarities with equations of hydrodynamical type, the system of equations 
	\eqref{equation les equations d'euler-lagrange, enfin!!!} is sometimes referred to as the \textit{hydrodynamical formulation 
	of quantum mechanics}. 
\end{remark}
	Note that the appearance of the time-dependant constant $c_{t}$ in 
	\eqref{equation les equations d'euler-lagrange, enfin!!!} is due to the $L^{2}$-orthogonal 
	decomposition \eqref{equation decomposition L2}\,.

\begin{remark}\label{remarque constante pas constante}
	By doing the change of variable $\phi':=\phi-\smallint c_{t}\,dt$ if necessary, one may 
	assume that the time-dependant constant $c_{t}$ of Proposition 
	\ref{proposition les Euler-Lagrange} is zero. 
\end{remark}
	As it is well known, if $c_{t}\equiv 0\,,$ then the system 
	\eqref{equation les equations d'euler-lagrange, enfin!!!} is equivalent to 
	the Schr\"{o}dinger equation for a quantum charged particle in an electromagnetic field:
	\begin{eqnarray}\label{equation schrodinger equation!!! yep!}
		i\hslash \,\dfrac{\partial \psi}{\partial t}=-\dfrac{\hslash^{2}}{2}\,\Delta\,\psi
		-\dfrac{\hslash}{i}\,g(X,\nabla\phi)+\dfrac{1}{2}\,\Big(-\dfrac{\hslash}{i}\,
		\textup{div}\,(X)+\|X\|^{2}\Big)\,\psi+V\psi\,,
	\end{eqnarray}
	where 
	\begin{eqnarray}\label{pas allemagne}
		\psi:=\sqrt{\rho}\,e^{-\frac{i}{\hslash}\phi}\,.
	\end{eqnarray}
	Using Remark \ref{remarque constante pas constante}, we can thus state the following 
	corollary.
\begin{corollary}\label{corollaire du vraie lagrangien comme on aime!}
	Let $\rho$ be a solution in $\mathcal{D}$ of the Euler-Lagrange equations associated to 
	the Lagrangian ${\mathcal{L}}_{}\,:\,T\mathcal{D}\rightarrow \mathbb{R}$ 
	(see \eqref{equation definition bis du Lagrangien}), with
	$\partial\rho/\partial t=\textup{div}\,\big(\rho\,(\nabla\phi+X)\big)\,.$
	Then the wave function associated to $\rho\,,$
	\begin{eqnarray}
		\psi:=\sqrt{\rho}\,e^{-\frac{i}{\hslash}\big(\phi-\smallint c_{t}dt\big)}\,,
	\end{eqnarray}
	(see \eqref{equation les equations d'euler-lagrange, enfin!!!} for the definition of $c_{t}$),
	satisfies the Schr\"{o}dinger equation \eqref{equation schrodinger equation!!! yep!}\,.
\end{corollary}
\begin{remark}\label{remarque sur l'extension de tau}
			For a smooth function $\psi\,:\,M\rightarrow \mathbb{C}\,,$ let us 
			denote by $[\psi]$ the complex line generated by $\psi$ in 
			the complex Hilbert space $\mathcal{H}:=L^{2}(M,\mathbb{C})$ (the latter being 
			endowed with its natural $L^{2}$-scalar product). Let us also consider the following map 
			\begin{eqnarray}
				T\,:\,T\mathcal{D}\rightarrow \mathbb{P}(\mathcal{H})\,,\,\,\,\,\,
				\,\,\,(\rho,\nabla\phi)\mapsto \big[\sqrt{\rho}\,e^{-\frac{i}{\hbar}\phi}\big]\,,
			\end{eqnarray}
			where $\mathbb{P}(\mathcal{H})$ denotes the complex projective space of 
			complex lines in $\mathcal{H}\,.$ \\
			As one may easily see, this map is well defined, and since  			
			\begin{eqnarray}
				\big[\sqrt{\rho}\,e^{-\frac{i}{\hslash}\big(\phi-\smallint c_{t}dt\big)}\big]=
				\big[\sqrt{\rho}\,e^{-\frac{i}{\hslash}\phi}\big]\,,
			\end{eqnarray} 
			Corollary \ref{corollaire du vraie lagrangien comme on aime!} implies that 
			$T$ maps solutions of the Euler-Lagrange equations 
			\eqref{equation les equations d'euler-lagrange, enfin!!!} to 
			solutions of the Schr\"{o}dinger equation \eqref{equation schrodinger equation!!! yep!}\,,
			projected on $\mathbb{P}(\mathcal{H})\,.$ 
\end{remark}

\section{Hamiltonian formulation}\label{section Hamiltonian}
\label{section Hamiltonian formulation}						
	In this section, we continue our study of the dynamics of a quantum particle initiated in 
	\S\ref{section Euler-Lagrange etc.}, but we will now focus on the Hamiltonian formulation. \\
	We will still assume that $(M,g)$ is a compact, connected and oriented Riemannian manifold, but for simplicity, 
	we will assume that the particle is only under the influence of a time-independent 
	potential $V\,:\,M\rightarrow \mathbb{R}\,.$\\

	Usually, the Hamiltonian formulation of a Lagrangian system is obtained by pulling back
	the canonical symplectic form of the cotangent bundle of the configuration manifold 
	via the Legendre transform (see \cite{Abraham-Marsden}). In our case, the configuration 
	manifold $\mathcal{D}$ being infinite dimensional, its
	cotangent bundle is no more a Fr\'echet manifold, rather a manifold 
	modelled on more general locally convex topological spaces, and we thus want to avoid it. \\
	To this end, we observe (see Corollary \ref{corollaire du vraie lagrangien comme on aime!})
	that the Lagrangian which describes a quantum particle under the influence 
	of a potential is given by 
	\begin{eqnarray}\label{definition du Lagrangien simplifie}
		{\mathcal{L}}(\rho,\nabla \phi)=\dfrac{1}{2}\int_{M}\,\|\nabla\phi_{}\|^{2}\,
			\rho\,\cdot d\textup{vol}_{g}-\int_{M}\,V\,\rho\,\cdot d\textup{vol}_{g}
			-\dfrac{\hslash^{2}}{2}\,\int_{M}\,\big\|\nabla(\sqrt{\rho}\,)\big\|^{2}\,\cdot 
			d\textup{vol}_{g}\,,\,\,\,\,\,\,\,\,\text{}
	\end{eqnarray}
	where $(\rho,\nabla\phi)\in \mathcal{D}\times \nabla C^{\infty}(M)\,.$\\
	This Lagrangian is of the form kinetic energy minus two potential terms, and,
	heuristically at least, this implies that the associated Lagrangian symplectic 
	form\footnote{Recall that the Lagrangian symplectic form $\omega_{{L}}\in \Omega^{2}(TM)$ 
	associated to a Lagrangian 
	$L\,:\,TM\rightarrow \mathbb{R}\,,$ is 
	the pull back via the Legendre transform $\mathbb{F}L\,:\,TM\rightarrow T^{*}M\,,$ 
	$\mathbb{F}L(u_{x})(v_{x}):=\frac{d}{dt}\vert_{0}L(u_{x}+tv_{x})\,,$
	of the canonical symplectic form $\omega=:-d\theta$ of the 
	cotangent bundle, where the canonical 1-form $\theta\in \Omega^{1}(T^{*}M)$ is defined, 
	for $A_{\alpha_{x}}\in T_{\alpha_{x}}T^{*}M\,,$ by $\theta_{\alpha_{x}}(A_{\alpha_{x}}):=
	\alpha_{x}\big((\pi^{T^{*}M})_{*_{\alpha_{x}}}A_{\alpha_{x}}\big)\,,$ 
	$\pi^{T^{*}M}\rightarrow M$ being the canonical projection, 
	see \cite{Abraham-Marsden}.} on $T\mathcal{D}$ 
	is uniquely determined by the metric 
	\begin{eqnarray}\label{equation definition la une forme}
		(g^{\mathcal{D}})_{\rho}\big((\rho,\nabla\phi),(\rho,\nabla\phi')\big):=\int_{M}\,
		g(\nabla \phi,\nabla\phi')\,\rho\,\cdot d\textup{vol}_{g}\,,
	\end{eqnarray}
	since potentials vanish under the Legendre transform. This motivates us, by mimicking the 
	finite dimensional construction, to define 
	the canonical 1-form $\Theta_{{\mathcal{L}}}$ on $T\mathcal{D}\cong \mathcal{D}\times
	\nabla C^{\infty}(M)$ via the formula :
	\begin{eqnarray}\label{equation je me suis gourré!!}
		(\Theta_{{\mathcal{L}}})_{(\rho,\nabla \phi)}(A_{(\rho,\nabla \phi)}):=(g^{\mathcal{D}})_{\rho}\Big(\nabla\phi,
		(\pi^{T\mathcal{D}})_{*_{(\rho,\nabla \phi)}}A_{(\rho,\nabla \phi)}\Big)\,,
	\end{eqnarray}
	where $A_{(\rho,\nabla \phi)}\in T_{(\rho,\nabla \phi)}T\mathcal{D}$ and where 
	$\pi^{T\mathcal{D}}\,:\,T\mathcal{D}\rightarrow \mathcal{D}$ denotes the canonical 
	projection.\\
	The right hand side of \eqref{equation je me suis gourré!!} is formally the pull back 
	of the canonical 1-form of the full cotangent bundle $T^{*}\mathcal{D}$
	via the Legendre transform associated to ${\mathcal{L}}\,.$ We will not explain this 
	point any further, but we will consider \eqref{equation je me suis gourré!!} as 
	the starting point of our study of the Hamiltonian description of a quantum particle. \\

	Our aim is now to compute explicitly the differential of $\Theta_{{\mathcal{L}}}\,,$ and to show that 
	$\Omega_{{\mathcal{L}}}:=-d \Theta_{{\mathcal{L}}}$ is a symplectic form on $T\mathcal{D}\,.$ To this end, we will use 
	the following identification 
	\begin{eqnarray}\label{identification espaces densites}
		T(T\mathcal{D})\cong \mathcal{D}\times \nabla C^{\infty}(M)
			\times\nabla C^{\infty}(M)\times\nabla C^{\infty}(M)\,,
	\end{eqnarray}
	the diffeomorphism being given by 
	\begin{eqnarray}
		\dfrac{d}{dt}\,\bigg\vert_{0}\,\big(\rho_{t},\nabla\phi+t\nabla\psi_{2}\big)\mapsto 
		(\rho_{0},\nabla\phi,\nabla\psi_{1},\nabla\psi_{2})\,,
	\end{eqnarray}
	where $\rho_{t}$ is a smooth curve in $\mathcal{D}$ satisfying 
	\begin{eqnarray}
		\dfrac{d}{dt}\,\bigg\vert_{0}\,\rho_{t}=\textup{div}\,(\rho_{0}\cdot \nabla\psi_{1})\,.
	\end{eqnarray}	
	Using \eqref{equation definition la une forme} and \eqref{identification espaces densites}, 
	it is clear that 
	\eqref{equation je me suis gourré!!} may be rewritten 
	\begin{eqnarray}\label{equation forme explicite de theta}
		(\Theta_{{\mathcal{L}}})_{(\rho,\nabla\phi)}\big(\rho,\nabla\phi,\nabla\psi_{1},\nabla\psi_{2}\big)
		=\int_{M}\,g(\nabla\phi,\nabla\psi_{1})\,\rho\cdot d\textup{vol}_{g}\,.
	\end{eqnarray}
	Our strategy to compute the differential of $\Theta_{{\mathcal{L}}}$ at a point $(\rho,\nabla\phi)\,,$ 
	will be to use the formula 
	\begin{eqnarray}\label{equation differential 1 form}
		(d\Theta_{{\mathcal{L}}})_{(\rho,\nabla\phi)}(X,Y)
		=X_{(\rho,\nabla\phi)}\big(\Theta_{{\mathcal{L}}}(Y)\big)-
		Y_{(\rho,\nabla\phi)}\big(\Theta_{{\mathcal{L}}}(X)\big)-
		(\Theta_{{\mathcal{L}}})_{(\rho,\nabla\phi)}([X,Y])\,,
	\end{eqnarray}
	where $X,Y$ are vector fields on $T\mathcal{D}\,.$ \\
	As the above formula is tensorial in $X$ and $Y\,,$ 
	we are free to choose $X$ and $Y$ arbitrary at a given point 
	$(\rho,\nabla\phi)\,,$ and to extend
	these vector fields as simply as possible elsewhere. A natural choice is to set, 
	for any $(\rho,\nabla\phi)\in T\mathcal{D}\,,$
	\begin{eqnarray}\label{equation definition X Y}
		X_{(\rho,\nabla\phi)}:=(\rho,\nabla\phi,\nabla\psi_{1},\nabla\psi_{2})\,\,\,\,\,
		\textup{and}\,\,\,\,\,
		Y_{(\rho,\nabla\phi)}:=(\rho,\nabla\phi,\nabla\alpha_{1},\nabla\alpha_{2})\,,
	\end{eqnarray}
	where $\nabla\psi_{1},\nabla\psi_{2},\nabla\alpha_{1},\nabla\alpha_{2}$ are held fixed.\\
	In view of \eqref{equation differential 1 form}, we now have to 
	compute $X_{(\rho,\nabla\phi)}\big(\Theta_{{\mathcal{L}}}(Y)\big)$ and 
	$(\Theta_{{\mathcal{L}}})_{(\rho,\nabla\phi)}([X,Y])\,,$ with $X$ and $Y$ as defined in 
	\eqref{equation definition X Y}\,.
\begin{lemma}\label{lemma derivée le long de X}
	We have :
	\begin{eqnarray}
		X_{(\rho,\nabla\phi)}\big(\Theta_{{\mathcal{L}}}(Y)\big)
			=\int_{M}\,g(\nabla\phi,\nabla\alpha_{1})\,
			\textup{div}(\rho\cdot\nabla\psi_{1})\cdot d\textup{vol}_{g}+\int_{M}
			\,g(\nabla\psi_{2},\nabla\alpha_{1})\,
			\rho\cdot d\textup{vol}_{g}\,.
	\end{eqnarray}
\end{lemma}
\begin{proof}
	Let $\rho_{t}$ be a curve in $\mathcal{D}$ satisfying 
	\begin{eqnarray}
		\rho_{0}=\rho\,\,\,\,\,\textup{and}\,\,\,\,\,
		\dfrac{\partial \rho_{t}}{\partial t}=\textup{div}\,(\rho_{t}\cdot\nabla\psi_{1})\,.
	\end{eqnarray}
	If $c(t):=(\rho_{t},\nabla\phi+t\nabla\psi_{2})\,,$ then
	\begin{eqnarray}
		\dot{c}(0)=(\rho,\nabla\phi,\nabla\psi_{1},\nabla\psi_{2})=X_{(\rho,\nabla\phi)}\,,
	\end{eqnarray}
	and thus, 
	\begin{eqnarray}
		&&X_{(\rho,\nabla\phi)}\big(\Theta_{{\mathcal{L}}}(Y)\big)\nonumber\\
		&=&\dfrac{d}{dt}\bigg\vert_{0}\,(\Theta_{{\mathcal{L}}})_{c(t)}(Y_{c(t)})
			=\dfrac{d}{dt}\bigg\vert_{0}\,\Theta_{{\mathcal{L}}}\big(\rho_{t},\nabla\phi+t\nabla\psi_{2},\nabla\alpha_{1},
			\nabla\alpha_{2}\big)\nonumber\\
		&=&\dfrac{d}{dt}\bigg\vert_{0}\,\bigg[\int_{M}\,g\big(\nabla\phi,\nabla\alpha_{1}\big)
			\rho_{t}\cdot d\textup{vol}_{g}+t\int_{M}\,g\big(\nabla\psi_{2},\nabla\alpha_{1}\big)\rho_{t}
			\cdot d\textup{vol}_{g}\bigg]\nonumber\\
		&=&\int_{M}\,g\big(\nabla\phi,\nabla\alpha_{1}\big)\,\textup{div}\,
			(\rho\cdot\nabla\psi_{1})\cdot d\textup{vol}_{g}+\int_{M}\,
			g\big(\nabla\psi_{2},\nabla\alpha_{1}\big)\rho\cdot d\textup{vol}_{g}\,.\,\,\,\,\,\,\,\textbf{}
	\end{eqnarray}
	The lemma follows. 
\end{proof}
	For the term $(\Theta_{{\mathcal{L}}})_{(\rho,\nabla\phi)}([X,Y])\,,$ we need to compute the Lie bracket $[X,Y]$ of $X$
	and $Y\,,$ and this can be done with a good description of the flow $\varphi^{X}_{t}$ of $X\,.$ \\
	This description may be obtained with the following map 
	\begin{eqnarray}
		D\,:\,\textup{Diff}(M)\rightarrow C^{\infty}(M)\,,
	\end{eqnarray}
	which is defined, for a $\varphi$ belonging to the group of all diffeomorphims
	$\textup{Diff}(M)\,,$ via the formula 
	\begin{eqnarray}\label{equation definition D}
		\varphi^{*}d\textup{vol}_{g}=D(\varphi)\cdot d\textup{vol}_{g}\,.
	\end{eqnarray}
	As a matter of notation, we shall write $D(\varphi)=\varphi^{*}d\textup{vol}_{g}/d\textup{vol}_{g}\,.$\\
	It may be shown that the map $D$ is smooth (see \cite{Hamilton}), and that 
	\begin{eqnarray}\label{equation proprietes de D}
		D(\varphi\circ \psi)=D(\varphi)\circ\psi\cdot D(\psi)\,,
	\end{eqnarray}
	where $\varphi,\psi\in \textup{Diff}(M)\,.$\\
	Observe also that if a diffeomorphism $\varphi$ preserves the orientation of $(M,d\textup{vol}_{g})\,,$ 
	then 
	$1/\textup{Vol}(M)\cdot D(\varphi)\in \mathcal{D}\,.$
\begin{lemma}\label{lemma le flowwww}
	The flow $\varphi^{X}_{t}$ of $X\,,$ is given, for $\rho\in\mathcal{D}$ and $\nabla\phi\in 
	\nabla C^{\infty}(M)\,,$ by
	\begin{eqnarray}\label{le flow}
		\varphi_{t}^{X}(\rho,\nabla\phi):=\bigg(\dfrac{1}{\textup{Vol}(M)}\cdot 
		D(\varphi\circ \varphi_{t}^{\nabla\psi_{1}}),\nabla\phi+t\nabla\psi_{2}\bigg)\,,
	\end{eqnarray}
	where $\varphi\in \textup{Diff}(M)$ is chosen such that $D(\varphi)=\textup{Vol}(M)\rho$ (such $\varphi$
	necessarily exits according to Moser's Theorem). 
\end{lemma}
\begin{proof}
	According to \eqref{equation definition D}, we have :
	\begin{eqnarray}\label{equation apareil photo}
		&&D(\varphi\circ\varphi_{t}^{\nabla\psi_{1}})\cdot d\textup{vol}_{g}
%			&=&(\varphi\circ \varphi_{t}^{\nabla\psi_{1}})^{*}d\textup{vol}_{g}\nonumber\\
			=(\varphi_{t}^{\nabla\psi_{1}})^{*}\varphi^{*}d\textup{vol}_{g}=
			(\varphi_{t}^{\nabla\psi_{1}})^{*}D(\varphi)\cdot d\textup{vol}_{g}\nonumber\\
		&\Rightarrow& D(\varphi\circ\varphi_{t}^{\nabla\psi_{1}})\cdot d\textup{vol}_{g}=
			(\varphi_{t}^{\nabla\psi_{1}})^{*}D(\varphi)\cdot d\textup{vol}_{g}\nonumber\\
		&\Rightarrow& \dfrac{d}{dt}D(\varphi\circ\varphi_{t}^{\nabla\psi_{1}})\cdot d\textup{vol}_{g}
			=\mathscr{L}_{\nabla\psi_{1}}\Big((\varphi_{t}^{\nabla\psi_{1}})^{*}D(\varphi)\cdot  
			d\textup{vol}_{g}\Big)\,,
	\end{eqnarray}
	and, in view of \eqref{equation proprietes de D},
	\begin{eqnarray}\label{equation photo 2}
		&&\mathscr{L}_{\nabla\psi_{1}}\Big((\varphi_{t}^{\nabla\psi_{1}})^{*}D(\varphi)\cdot  
			d\textup{vol}_{g}\Big)\nonumber\\
		&=&\mathscr{L}_{\nabla\psi_{1}}\Big((D(\varphi)\circ\varphi_{t}^{\nabla\psi_{1}})\cdot 
			D\big(\varphi_{t}^{\nabla\psi_{1}}\big)\cdot  
			d\textup{vol}_{g}\Big)\nonumber\\
		&=&\mathscr{L}_{\nabla\psi_{1}}\Big(D(\varphi\circ\varphi_{t}^{\nabla\psi_{1}})\cdot 
			d\textup{vol}_{g}\Big)\nonumber\\
		&=&\Big(g\big(\nabla\psi_{1},\nabla D(\varphi\circ\varphi_{t}^{\nabla\psi_{1}})\big)+
			D(\varphi\circ\varphi_{t}^{\nabla\psi_{1}})\,\textup{div}(\nabla\psi_{1})\Big)\cdot 
			d\textup{vol}_{g}\nonumber\\
		&=&\textup{div}\,\Big(D(\varphi\circ\varphi_{t}^{\nabla\psi_{1}})\cdot \nabla\psi_{1}\Big)\cdot 
			d\textup{vol}_{g}\,.
	\end{eqnarray}
	Collecting \eqref{equation apareil photo} and \eqref{equation photo 2}, we thus get 
	\begin{eqnarray}\label{equation propriete de D}
		\dfrac{d}{dt}D(\varphi\circ \varphi_{t}^{\nabla\psi_{1}})=
		\textup{div}\,\Big(D(\varphi\circ\varphi_{t}^{\nabla\psi_{1}})\cdot \nabla\psi_{1}\Big)\,, 
	\end{eqnarray}
	from which we see, having in mind the identification \eqref{identification espaces densites}, that 
	\begin{eqnarray}\label{equation ou est timot}
		&&\dfrac{d}{dt}\bigg(\dfrac{1}{\textup{Vol}(M)}\cdot 
			D(\varphi\circ \varphi_{t}^{\nabla\psi_{1}}),\nabla\phi+t\nabla\psi_{2}\bigg)\nonumber\\
		&=&\bigg(\dfrac{1}{\textup{Vol}(M)}\cdot 
			D(\varphi\circ \varphi_{t}^{\nabla\psi_{1}}),\nabla\phi+t\nabla\psi_{2},
			\nabla\psi_{1},\nabla\psi_{2}\bigg)\,.
	\end{eqnarray}
	Equation \eqref{equation ou est timot} exactly means that $\varphi_{t}^{X}\,,$ such as defined in 
	\eqref{le flow}, is the flow of $X\,.$ The lemma follows. 
\end{proof}
	We are now almost able to compute the Lie bracket $[X,Y]\,.$ But for this, we still need, 
	for $\rho\in \mathcal{D}\,,$ the following continuous map of Fr\'echet spaces  
	\begin{eqnarray}\label{equation HH decomposition}
		\mathbb{P}_{\rho}\,:\,
		\left \lbrace
			\begin{array}{cc}
				\mathfrak{X}(M)=\mathfrak{X}_{d\textup{vol}_{g}}(M)\oplus \rho\nabla C^{\infty}(M)
				\rightarrow \nabla C^{\infty}(M)\,,\\
				X=\overline{X}+\rho\nabla\phi\mapsto \nabla\phi\,,
			\end{array}
		\right. 
	\end{eqnarray}
	where $\overline{X}\in \mathfrak{X}_{d\textup{vol}_{g}}(M)=
	\{Z\in\mathfrak{X}(M)\,\vert\,\textup{div}\,(Z)=0\}\,,$ and where the topological direct sum 
	$\mathfrak{X}(M)=\mathfrak{X}_{d\textup{vol}_{g}}(M)
	\oplus \rho\nabla C^{\infty}(M)$ is simply 
	a slight generalisation of the Helmholtz-Hodge decomposition (see \eqref{equation HH} and 
	\cite{Molitor-Yang-Mills} for a proof of this generalization).  
\begin{remark}\label{j'en ai marre de ce papier!!}
	Using Stokes' Theorem, it is easy to show the following convenient formula : 
	\begin{eqnarray}
		\int_{M}\,g\big(\nabla\phi,\mathbb{P}_{\rho}(\rho X)\big)\,\rho\cdot d\textup{vol}_{g}=
		\int_{M}\,g(\nabla\phi,X)\,\rho\cdot d\textup{vol}_{g}\,,
	\end{eqnarray}
	where $\phi\in C^{\infty}(M)$ and where $X\in \mathfrak{X}(M)\,.$
\end{remark}
\begin{lemma}\label{lemma le lie bracket!!!}
	For $\rho\in \mathcal{D}$ and $\nabla\phi\in \nabla C^{\infty}(M)\,,$ we have :
	\begin{eqnarray}
		[X,Y]_{(\rho,\nabla\phi)}=\Big(\rho,\nabla\phi,
		\mathbb{P}_{\rho}\big(\rho\,[\nabla\alpha_{1},\nabla\psi_{1}]\big),0\Big)\,.
	\end{eqnarray}
\end{lemma}
\begin{proof}
	Let us choose $\varphi,\psi_{t},\beta_{t,s}\in \textup{Diff}(M)$ such that 
	\begin{eqnarray}\label{equation que dire manque de logique}
		D(\varphi)=\textup{Vol}(M)\cdot \rho\,,\,\,\,D(\psi_{t})=
		D(\varphi\circ \varphi_{t}^{\nabla\psi_{1}})\,,\,\,\,
		D(\beta_{t,s})=D(\psi_{t}\circ\varphi_{s}^{\nabla\alpha_{1}})\,.
	\end{eqnarray}
	According to Lemma \ref{lemma le flowwww}, we have 
	\begin{eqnarray}\label{equation il fait beau, bientot le japon}
		&&[X,Y]_{(\rho,\nabla\phi)}=\dfrac{d}{dt}\bigg\vert_{0}\bigg(
			(\varphi_{-t}^{X})_{*_{\varphi_{t}^{X}(\rho,\nabla\phi)}}Y_{\varphi_{t}^{X}
			(\rho,\nabla\phi)}\bigg)\nonumber\\
		&=&\dfrac{d}{dt}\bigg\vert_{0}\dfrac{d}{ds}\bigg\vert_{0}\Big(
			\varphi_{-t}^{X}\circ \varphi_{s}^{Y}\circ \varphi_{t}^{X}\Big)(\rho,\nabla\phi)\nonumber\\
		&=&\dfrac{d}{dt}\bigg\vert_{0}\dfrac{d}{ds}\bigg\vert_{0}\Big(
			\varphi_{-t}^{X}\circ \varphi_{s}^{Y}\Big)\bigg(\dfrac{1}{\textup{Vol}(M)}\cdot D(\varphi\circ
			\varphi_{t}^{\nabla\psi_{1}}),\nabla\phi+t\nabla\psi_{2}\bigg)\nonumber\\
		&=&\dfrac{d}{dt}\bigg\vert_{0}\dfrac{d}{ds}\bigg\vert_{0}
			\varphi_{-t}^{X}\bigg(\dfrac{1}{\textup{Vol}(M)}\cdot D(\psi_{t}\circ
			\varphi_{s}^{\nabla\alpha_{1}}),\nabla\phi+t\nabla\psi_{2}+s\nabla\alpha_{2}\bigg)\nonumber\\
		&=&\dfrac{d}{dt}\bigg\vert_{0}\dfrac{d}{ds}\bigg\vert_{0}
			\bigg(\dfrac{1}{\textup{Vol}(M)}\cdot D(\beta_{t,s}\circ
			\varphi_{-t}^{\nabla\psi_{1}}),\nabla\phi+
			t\nabla\psi_{2}+s\nabla\alpha_{2}-t\nabla\psi_{2}\bigg)\nonumber\\
		&=&\dfrac{d}{dt}\bigg\vert_{0}\dfrac{d}{ds}\bigg\vert_{0}
			\bigg(\dfrac{1}{\textup{Vol}(M)}\cdot D(\beta_{t,s}\circ
			\varphi_{-t}^{\nabla\psi_{1}}),\nabla\phi+s\nabla\alpha_{2}\bigg)\,.
	\end{eqnarray}
	From \eqref{equation il fait beau, bientot le japon}, we already see that the bracket 
	$[X,Y]_{(\rho,\nabla\phi)}$ is of the form $(\rho,\nabla\phi,\ast,0)\,,$ where ``\,$\ast$\," has to be  
	determined by computing the derivatives of $D(\beta_{t,s}\circ\varphi_{-t}^{\nabla\psi_{1}})$ 
	with respect to $s$ and $t\,,$ and by putting it in a divergence form. \\
	Using \eqref{equation proprietes de D} and \eqref{equation que dire manque de logique}\,, we see that
	\begin{eqnarray}
		D(\beta_{t,s}\circ \varphi_{-t}^{\nabla\psi_{1}})&=&D(\beta_{t,s})\circ \varphi_{-t}^{\nabla\psi_{1}}\cdot
			D(\varphi_{-t}^{\nabla\psi_{1}})=D(\psi_{t}\circ \varphi_{s}^{\nabla\alpha_{1}})
			\circ \varphi_{-t}^{\nabla\psi_{1}}\cdot
			D(\varphi_{-t}^{\nabla\psi_{1}})\nonumber\\
		&=&D(\psi_{t})\circ \varphi_{s}^{\nabla\alpha_{1}}\circ \varphi_{-t}^{\nabla\psi_{1}}\cdot 
			D(\varphi_{s}^{\nabla{\alpha_{1}}})\circ
			\varphi_{-t}^{\nabla\psi_{1}}\cdot D(\varphi_{-t}^{\nabla\psi_{1}})\nonumber\\
		&=&D(\varphi\circ\varphi_{t}^{\nabla\psi_{1}})\circ 
			\varphi_{s}^{\nabla\alpha_{1}}\circ \varphi_{-t}^{\nabla\psi_{1}}\cdot 
			D(\varphi_{s}^{\nabla{\alpha_{1}}})\circ
			\varphi_{-t}^{\nabla\psi_{1}}\cdot D(\varphi_{-t}^{\nabla\psi_{1}})\nonumber\\
		&=&\Big(D(\varphi\circ\varphi_{t}^{\nabla\psi_{1}})\circ 
			\varphi_{s}^{\nabla\alpha_{1}}\cdot 
			D(\varphi_{s}^{\nabla{\alpha_{1}}})\Big)\circ
			\varphi_{-t}^{\nabla\psi_{1}}\cdot D(\varphi_{-t}^{\nabla\psi_{1}})\nonumber\\
		&=&D(\varphi\circ\varphi_{t}^{\nabla\psi_{1}}\circ \varphi_{s}^{\nabla\alpha_{1}})\circ
			\varphi_{-t}^{\nabla\psi_{1}}\cdot D(\varphi_{-t}^{\nabla\psi_{1}})\nonumber\\
		&=&D(\varphi\circ\varphi_{t}^{\nabla\psi_{1}}\circ \varphi_{s}^{\nabla\alpha_{1}}\circ 
			\varphi_{-t}^{\nabla\psi_{1}})\,, 
	\end{eqnarray}
	and thus, 
	\begin{eqnarray}\label{equation faut payer le canada}
		&&\dfrac{d}{dt}\bigg\vert_{0}\dfrac{d}{ds}\bigg\vert_{0}\,\dfrac{1}{\textup{Vol}(M)}\cdot 
			D(\beta_{t,s}\circ\varphi_{-t}^{\nabla\psi_{1}})\nonumber\\
		&=&\dfrac{d}{dt}\bigg\vert_{0}\dfrac{d}{ds}\bigg\vert_{0}\,\dfrac{1}{\textup{Vol}(M)}\cdot 
			D(\varphi\circ\varphi_{t}^{\nabla\psi_{1}}\circ\varphi_{s}^{\nabla\alpha_{1}}\circ
			\varphi_{-t}^{\nabla\psi_{1}})\nonumber\\
		&=&\dfrac{d}{dt}\bigg\vert_{0}\dfrac{d}{ds}\bigg\vert_{0}\,\rho\circ 
			\varphi_{t}^{\nabla\psi_{1}}\circ\varphi_{s}^{\nabla\alpha_{1}}\circ
			\varphi_{-t}^{\nabla\psi_{1}}\cdot 
			D(\varphi_{t}^{\nabla\psi_{1}}\circ\varphi_{s}^{\nabla\alpha_{1}}\circ
			\varphi_{-t}^{\nabla\psi_{1}})\nonumber\\
		&=&\dfrac{d}{dt}\bigg\vert_{0}\dfrac{d}{ds}\bigg\vert_{0}\,\rho\circ 
			\varphi_{t}^{\nabla\psi_{1}}\circ\varphi_{s}^{\nabla\alpha_{1}}\circ
			\varphi_{-t}^{\nabla\psi_{1}}+\rho\cdot 
			\dfrac{d}{dt}\bigg\vert_{0}\dfrac{d}{ds}\bigg\vert_{0}\,
			D(\varphi_{t}^{\nabla\psi_{1}}\circ\varphi_{s}^{\nabla\alpha_{1}}\circ
			\varphi_{-t}^{\nabla\psi_{1}})\nonumber\\
		&=& g\big(\nabla\rho,[\nabla\alpha_{1},\nabla\psi_{1}]\big)+\rho\cdot 
			\dfrac{d}{dt}\bigg\vert_{0}\dfrac{d}{ds}\bigg\vert_{0}\,
			\big(\varphi_{t}^{\nabla\psi_{1}}\circ\varphi_{s}^{\nabla\alpha_{1}}\circ
			\varphi_{-t}^{\nabla\psi_{1}}\big)^{*} d\textup{vol}_{g}/d\textup{vol}_{g}\nonumber\\
		&=& g\big(\nabla\rho,[\nabla\alpha_{1},\nabla\psi_{1}]\big)+\rho\cdot 
			\dfrac{d}{dt}\bigg\vert_{0}
			\mathscr{L}_{(\varphi_{t}^{\nabla\psi_{1}})_{*_{\varphi_{-t}^{\nabla\psi_{1}}}}
			(\nabla\alpha_{1})_{\varphi_{-t}^{\nabla\psi_{1}}}}\big(d\textup{vol}_{g}\big)
			/d\textup{vol}_{g}\nonumber\\
		&=& g\big(\nabla\rho,[\nabla\alpha_{1},\nabla\psi_{1}]\big)+\rho\cdot 
			\dfrac{d}{dt}\bigg\vert_{0}
			\textup{div}\,\Big({(\varphi_{t}^{\nabla\psi_{1}})_{*_{\varphi_{-t}^{\nabla\psi_{1}}}}
			(\nabla\alpha_{1})_{\varphi_{-t}^{\nabla\psi_{1}}}}\Big)\nonumber\\
		&=& g\big(\nabla\rho,[\nabla\alpha_{1},\nabla\psi_{1}]\big)
			+\rho\cdot \textup{div}\,\big([\nabla\alpha_{1},\nabla\psi_{1}]\big)\nonumber\\
		&=&\textup{div}\,\big(\rho\cdot [\nabla\alpha_{1},\nabla\psi_{1}]\big)=
			\textup{div}\,\big(\rho\cdot \mathbb{P}_{\rho}\big(
			\rho\cdot [\nabla\alpha_{1},\nabla\psi_{1}]\big)\big)\,.
	\end{eqnarray}
	The lemma follows. 
\end{proof}
\begin{proposition}\label{proposition omega}
	The form $\Omega_{{\mathcal{L}}}:=-d\Theta_{{\mathcal{L}}}$ (see \eqref{equation forme explicite de theta} for the definition of 
	$\Theta_{{\mathcal{L}}}$), is a symplectic form on $T\mathcal{D}\,,$ and for $\rho\in \mathcal{D}$ and 
	$\nabla\phi\in \nabla C^{\infty}(M)\,,$
	\begin{eqnarray}\label{equation forme symplectic enfin!}
		&&(\Omega_{{\mathcal{L}}})_{(\rho,\nabla\phi)}\big((\rho,\nabla\phi,\nabla\psi_{1},\nabla\psi_{2}),
			(\rho,\nabla\phi,\nabla\alpha_{1},\nabla\alpha_{2})\big)\nonumber\\
		&=&\int_{M}\,g(\nabla\psi_{1},\nabla\alpha_{2})\,\rho\cdot d\textup{vol}_{g}-
			\int_{M}\,g(\nabla\alpha_{1},\nabla\psi_{2})\,\rho\cdot d\textup{vol}_{g}\,,
	\end{eqnarray}
	where $\nabla\psi_{1},\nabla\psi_{2},\nabla\alpha_{1},\nabla\alpha_{2}\in 
	\nabla C^{\infty}(M)\,.$
\end{proposition}
\begin{proof}
	The fact that $\Omega_{{\mathcal{L}}}$ is a symplectic form, i.e., that $\Omega_{{\mathcal{L}}}$ is non-degenerate (the 
	closedness being clear), is a simple consequence of 
	formula \eqref{equation forme symplectic enfin!} that we are now going to show. 
	
	Equation \eqref{equation differential 1 form}, together with Lemma \ref{lemma derivée le long de X} 
	and Lemma \ref{lemma le lie bracket!!!}, yield 
	\begin{eqnarray}\label{equation il faut du pain pour midi}
		&&(\Omega_{{\mathcal{L}}})_{(\rho,\nabla\phi)}\big((\rho,\nabla\phi,\nabla\psi_{1},\nabla\psi_{2}),
			(\rho,\nabla\phi,\nabla\alpha_{1},\nabla\alpha_{2})\big)\nonumber\\
		&=&-(d\Theta_{{\mathcal{L}}})_{(\rho,\nabla\phi)}\big(X_{(\rho,\nabla\phi)},Y_{(\rho,\nabla\phi)}\big)\nonumber\\
		&=& -X_{(\rho,\nabla\phi)}\big(\Theta_{{\mathcal{L}}}(Y)\big)+Y_{(\rho,\nabla\phi)}\big(\Theta_{{\mathcal{L}}}(X)\big)+
			(\Theta_{{\mathcal{L}}})_{(\rho,\nabla\phi)}([X,Y])\,,\nonumber\\
		&=&\int_{M}\,g(\nabla\psi_{1},\nabla\alpha_{2})\,\rho\cdot d\textup{vol}_{g}-
			\int_{M}\,g(\nabla\alpha_{1},\nabla\psi_{2})\,\rho\cdot d\textup{vol}_{g}\nonumber\\
		&&+\int_{M}\,g(\nabla\psi_{1},\nabla\phi)\,\textup{div}\,
			\big(\rho\cdot \nabla\alpha_{1}\big)\cdot d\textup{vol}_{g}
			-\int_{M}\,g(\nabla\alpha_{1},\nabla\phi)\,\textup{div}\,
			\big(\rho\cdot \nabla\psi_{1}\big)\cdot d\textup{vol}_{g}\nonumber\\
		&&-\int_{M}\,g\big(\nabla\phi,\mathbb{P}_{\rho}\big(\rho\cdot 
			[\nabla\psi_{1},\nabla\alpha_{1}]\big)\big)
			\,\rho\cdot d\textup{vol}_{g}\,.
	\end{eqnarray}
	Clearly, we have to show that the last two lines in \eqref{equation il faut du pain pour midi} 
	vanish. \\
	Using Remark \ref{j'en ai marre de ce papier!!}, one may rewrite the last term in 
	\eqref{equation il faut du pain pour midi} as 
	\begin{eqnarray}
		\int_{M}\,g\big(\nabla\phi,\mathbb{P}_{\rho}\big(\rho\cdot 
			[\nabla\psi_{1},\nabla\alpha_{1}]\big)\big)
			\,\rho\cdot d\textup{vol}_{g}
			=\int_{M}\,g\big(\nabla\phi,[\nabla\psi_{1},\nabla\alpha_{1}]\big)\rho \cdot 
			d\textup{vol}_{g}\,.
	\end{eqnarray}
	Using this last equation, one observes that the last three terms in 
	\eqref{equation il faut du pain pour midi} may be rewritten :
	\begin{eqnarray}
		&&\int_{M}\,g(\nabla\psi_{1},\nabla\phi)\,\textup{div}\,
			\big(\rho\cdot \nabla\alpha_{1}\big)\cdot d\textup{vol}_{g}-
			\int_{M}\,g(\nabla\alpha_{1},\nabla\phi)\,\textup{div}\,
			\big(\rho\cdot \nabla\psi_{1}\big)\cdot d\textup{vol}_{g}\nonumber\\
		&&-\int_{M}\,g\big(\nabla\phi,
			[\nabla\psi_{1},\nabla\alpha_{1}]\big)
			\,\rho\cdot d\textup{vol}_{g}\,.\nonumber\\
		&=&\int_{M}\,\bigg(-g\Big(\nabla\alpha_{1},\nabla g(\nabla\psi_{1},\nabla\phi)\Big)+
			g\Big(\nabla\psi_{1},\nabla g(\nabla\alpha_{1},\nabla\phi)\Big)
			-g(\nabla\phi,[\nabla\psi_{1},\nabla\alpha_{1}])\bigg)\,\rho\cdot d\textup{vol}_{g}\,\,\,\,\,\,
			\textbf{}\nonumber\\
		&=&\int_{M}\,\Big(-(\nabla\alpha_{1})\, d\phi(\nabla\psi_{1})+(\nabla\psi_{1}) \,
			d\phi(\nabla\alpha_{1})-d\phi([\nabla\psi_{1},\nabla\alpha_{1}])\Big)\,\rho\cdot 
			d\textup{vol}_{g}\nonumber\\
		&=&\int_{M}\,d(d\phi)(\nabla\psi_{1},\nabla\alpha_{1})\,\rho\cdot d\textup{vol}_{g}=0\,.
	\end{eqnarray}
	The proposition follows.
\end{proof}

	With such simple expression for the symplectic form $\Omega_{{\mathcal{L}}}$ (see 
	\eqref{equation forme symplectic enfin!}), it is possible the compute 
	explicitly the symplectic gradient of interesting functions, as well as their Poisson brackets. 
	Indeed, we define, 
	for $F\,:\,TM\rightarrow 
	\mathbb{R}\,,$ the following function on $T\mathcal{D}\,:$
	\begin{eqnarray}\label{equation definition de F chapeau}
		\widehat{F}(\rho,\nabla\phi):=\int_{M}\,F(\nabla\phi)\,\rho\cdot d\textup{vol}_{g}\,.
	\end{eqnarray}
	We also denote by ${\mathcal{H}}\,:\,T\mathcal{D}\rightarrow \mathbb{R}\,,$ 
	the Hamiltonian associated, via the Legendre transform, 
	to the Lagrangian ${\mathcal{L}}$
\footnote{Recall that if $L\,:\,TM\rightarrow\mathbb{R}$ is a Lagrangian defined on a manifold 
	$M\,,$ then its associated Hamiltonian $H\,:\,TM\rightarrow \mathbb{R}$ is the function defined, 
	for $u_{x}\in T_{x}M\,,$ by $H(u_{x}):=\mathbb{F}L(u_{x})(u_{x})-L(u_{x})\,,$
	 where $\mathbb{F}L\,:\,TM\rightarrow T^{*}M$ is the Legendre transform of $L\,.$}
	 :
	\begin{eqnarray}\label{equation definition du l'hamiltonien}
		{\mathcal{H}}(\rho,\nabla\phi):=\int_{M}\,\Big(\dfrac{1}{2}\,\|\nabla\phi\|^{2}+V\Big)\,
		\rho\cdot d\textup{vol}_{g}+\frac{\hslash^{2}}{2}\,
		\int_{M}\,\|\nabla\big(\sqrt{\rho}\big)\|^{2}\cdot 
		d\textup{vol}_{g}\,.
	\end{eqnarray}
	We shall denote by $X_{\widehat{F}}$ and $X_{{\mathcal{H}}}$ the symplectic gradients associated to 
	$\widehat{F}$ and ${\mathcal{H}}$ via the symplectic form $\Omega_{{\mathcal{L}}}$ (recall that these two vector 
	fields are defined on $T\mathcal{D}$ via the relations 
	$\Omega_{{\mathcal{L}}}(X_{\widehat{F}},\,.\,)=d\widehat{F}$ and $\Omega_{{\mathcal{L}}}(X_{{\mathcal{H}}},\,.\,)=d{\mathcal{H}})\,.$\\
	\textbf{}\\
	On $T\mathcal{D}\,,$ we shall use the Poisson bracket $\{\,.\,,\,.\,\}_{\mathcal{L}}$  
	associated to the symplectic 
	form $\Omega_{\mathcal{L}}$ (of course, this Poisson bracket is only defined 
	for functions having a symplectic gradient), and on $TM$ we shall use 
	the Poisson bracket, denoted $\{\,.\,,\,.\,\}_{L}\,,$ canonically associated to the Lagrangian 
	$L(u_{x}):=1/2\cdot g(u_{x},u_{x})-V(x)\,.$

\begin{proposition}\label{proposition symplectic gradients}
	For $F,G\,:\,TM\rightarrow \mathbb{R}\,,$
	$\rho\in \mathcal{D}$ and $\nabla\phi\in \nabla 
	C^{\infty}(M)\,,$ we have :
		\begin{enumerate}
			\item $(X_{{\mathcal{H}}} )_{(\rho,\nabla\phi)}=\bigg(\rho,\nabla\phi,\nabla\phi,
				\nabla\bigg[\dfrac{1}{2}\,\|\nabla\phi\|^{2}+V-\dfrac{\hslash^{2}}{2}\,
				\dfrac{\triangle\,(\sqrt{\rho})}{\sqrt{\rho}}\bigg]\bigg)\,,$
			\item $(X_{\widehat{F}})_{(\rho,\nabla\phi)}=\Big(\rho,\nabla\phi,\mathbb{P}_{\rho}
				\big(\rho\,(\pi_{*}^{TM}\circ X_{F}\circ \nabla\phi\big),
				\nabla\big(F(\nabla\phi)\big)\Big)\,,$
			\item $ \{\widehat{F},\widehat{G}\}_{\mathcal{L}}=-\widehat{\{F,G\}_{L}}\,.$
%			\item $\{\widehat{f},{\mathcal{H}}\}_{\mathcal{L}}=-\widehat{df}\,.$
		\end{enumerate}
\end{proposition}
	We will show Proposition \ref{proposition symplectic gradients} with a series of Lemmas. 
\begin{lemma}\label{lemme legradient symp de l'hami}
	For $\rho\in \mathcal{D}$ and $\nabla\phi\in \nabla 
	C^{\infty}(M)\,,$ we have :
	\begin{eqnarray}\label{eqnation autre definition du lagrangian vector field}
		(X_{{\mathcal{H}}} )_{(\rho,\nabla\phi)}=\bigg(\rho,\nabla\phi,\nabla\phi,
				\nabla\bigg[\dfrac{1}{2}\,\|\nabla\phi\|^{2}+V-\dfrac{\hslash^{2}}{2}\,
				\dfrac{\triangle\,(\sqrt{\rho})}{\sqrt{\rho}}\bigg]\bigg)\,.
	\end{eqnarray}
\end{lemma}
\begin{proof}
	
	We will use the vector field $X\in \mathfrak{X}(T\mathcal{D})$ introduced in 
	\eqref{equation definition X Y}, and especially its flow $\varphi_{t}^{X}$ which is
	given in Lemma \ref{lemma le flowwww}.\\
	We have :
	\begin{eqnarray}
		&&(d\mathcal{H})_{(\rho,\nabla\phi)}X_{(\rho,\nabla\phi)}
			= \dfrac{d}{dt}\bigg\vert_{0}\,(\mathcal{H}\circ\varphi_{t}^{X})(\rho,\nabla\phi)\nonumber\\
		&=&\dfrac{d}{dt}\bigg\vert_{0}\,\mathcal{H}\bigg(\dfrac{1}{\textup{Vol}(M)}\cdot D(\varphi\circ 
			\varphi_{t}^{\nabla\psi_{1}}),\nabla\phi+t\nabla\psi_{2}\bigg)\nonumber\\
		&=&\dfrac{d}{dt}\bigg\vert_{0}\,\,\bigg[\int_{M}\bigg(\dfrac{1}{2}\,
			\|\nabla\phi+t\nabla\psi_{2}\|^{2}
			+V\bigg)\,\dfrac{1}{\textup{Vol}(M)}\,D(\varphi\circ\varphi_{t}^{\nabla\psi_{1}})
			\cdot d\textup{vol}_{g}\nonumber\\
		&&\,\,\,\,\,\,
			\,+\dfrac{\hslash^{2}}{2}\int_{M}\bigg\|\nabla\bigg(\sqrt{\dfrac{1}{\textup{Vol}(M)}D(\varphi\circ
			\varphi_{t}^{\nabla\psi_{1}})}\bigg)\bigg\|^{2}\cdot d\textup{vol}_{g}
			\,\,\bigg]\nonumber\\
		&=& \dfrac{d}{dt}\bigg\vert_{0}\,\int_{M}\,\Big(\dfrac{1}{2}\,\|\nabla\phi\|^{2}+
			t\,g(\nabla\phi,\nabla\psi_{2})+\dfrac{t^{2}}{2}\,\|\nabla\psi_{2}\|^{2}
			+V\Big)\,\rho_{t}\cdot d\textup{vol}_{g}\nonumber\\
		&&\,\,\,\,\,\,\,+\dfrac{d}{dt}\bigg\vert_{0}\,\dfrac{\hslash^{2}}{2}\,\int_{M}\,
			\|\nabla\big(\sqrt{\rho_{t}}\,\big)\|^{2}\cdot d\textup{vol}_{g}\,,
	\end{eqnarray}
	where $\rho_{t}:=1/\textup{Vol}(M)\cdot D(\varphi\circ\varphi_{t}^{\nabla\psi_{1}})\,.$\\
	But, according to \eqref{equation propriete de D},
	\begin{eqnarray}\label{equation derivee de rhooooo}
		\dfrac{\partial \rho_{t}}{\partial t}=\textup{div}\,(\rho_{t}\cdot \nabla\psi_{1})\,,
	\end{eqnarray}
	and thus, 
	\begin{eqnarray}\label{equation calcul des gradients symplectiques!}
		&&(d\mathcal{H})_{(\rho,\nabla\phi)}X_{(\rho,\nabla\phi)}\nonumber\\
		&=&\dfrac{1}{2}\int_{M}\,\|\nabla\phi\|^{2}\,\textup{div}\,(\rho\cdot\nabla\psi_{1})\cdot 
			d\textup{vol}_{g}+\int_{M}\,g(\nabla\phi,\nabla\psi_{2})\,\rho\cdot d\textup{vol}_{g}\nonumber\\
		&&+\int_{M}V\,\textup{div}\,(\rho\cdot\nabla\psi_{1})\cdot d\textup{vol}_{g}+
			\dfrac{d}{dt}\bigg\vert_{0}\,\dfrac{\hslash^{2}}{2}\int_{M}\,
			\|\nabla\big(\sqrt{\rho_{t}}\,\big)\|^{2}\cdot d\textup{vol}_{g}\,.
	\end{eqnarray}
	Let us compute the last term in \eqref{equation calcul des gradients symplectiques!} :
	\begin{eqnarray}\label{equation immagration bureau}
		&&\dfrac{d}{dt}\bigg\vert_{0}\,\int_{M}\,
			\|\nabla\big(\sqrt{\rho_{t}}\,\big)\|^{2}\cdot d\textup{vol}_{g}
			=2\int_{M}\,g\Big(\nabla\dfrac{\partial}{\partial t}\bigg\vert_{0}\sqrt{\rho_{t}},
			\nabla\big(\sqrt{\rho}\big)\Big)\cdot d\textup{vol}_{g}\nonumber\\
		&=&\int_{M}\,g\Big(\nabla\Big[\dfrac{1}{\sqrt{\rho}}\,\textup{div}\,(\rho\cdot \nabla\psi_{1})
			\Big],\nabla\big(\sqrt{\rho}\,\big)\Big)\cdot d\textup{vol}_{g}\nonumber\\
		&=&\int_{M}\,g\Big(\textup{div}\,(\rho\cdot\nabla\psi_{1})\,\big(-\dfrac{1}{\rho}\cdot 
			\dfrac{1}{2\sqrt{\rho}}\,\nabla\rho\big)+\dfrac{1}{\sqrt{\rho}}\,\nabla\,
			\textup{div}\,(\rho\cdot\nabla\psi_{1})\,,
			\nabla\big(\sqrt{\rho}\,\big)\Big)\cdot d\textup{vol}_{g}\nonumber\\
		&=&-\int_{M}\,g\big(\nabla\rho,\nabla(\sqrt{\rho}\,)\big)\cdot \textup{div}\,
			(\rho\cdot \nabla\psi_{1})\dfrac{1}{\rho}\cdot\dfrac{1}{2\sqrt{\rho}}
			\cdot d\textup{vol}_{g}+\int_{M}\,g\big(\nabla\textup{div}\,(\rho\cdot\nabla\psi_{1}),\nabla(\sqrt{\rho}\,)\big)\,
			\dfrac{1}{\sqrt{\rho}}\cdot d\textup{vol}_{g}\nonumber\\
		&=&-\int_{M}\,\|\nabla\rho\|^{2}\cdot \textup{div}\,
			(\rho\cdot \nabla\psi_{1})\,\dfrac{1}{4}\cdot\dfrac{1}{\rho^{2}}
			\cdot d\textup{vol}_{g}+\int_{M}\,g\big(\nabla\textup{div}\,(\rho\cdot\nabla\psi_{1}),\nabla\rho\big)\,
			\dfrac{1}{2\rho}\cdot d\textup{vol}_{g}\nonumber\\
		&=&-\int_{M}\,\|\nabla\rho\|^{2}\cdot \textup{div}\,
			(\rho\cdot \nabla\psi_{1})\,\dfrac{1}{4}\cdot\dfrac{1}{\rho^{2}}
			\cdot d\textup{vol}_{g}-\int_{M}\,\textup{div}\,(\rho\cdot\nabla\psi_{1})\cdot \mathscr{L}_{\nabla\rho}
			\Big(\dfrac{1}{2\rho}\cdot d\textup{vol}_{g}\Big)\nonumber\\
		&=&-\int_{M}\,\|\nabla\rho\|^{2}\cdot \textup{div}\,
			(\rho\cdot \nabla\psi_{1})\,\dfrac{1}{4}\cdot\dfrac{1}{\rho^{2}}
			\cdot d\textup{vol}_{g}
			-\int_{M}\,\textup{div}\,(\rho\cdot\nabla\psi)\,\,g\big(\nabla\rho,\nabla\big(\dfrac{1}{2\rho}\big)
			\big)\cdot d\textup{vol}_{g}\nonumber\\
		&&-\int_{M}\,\textup{div}\,(\rho\cdot\nabla\psi_{1})\,\dfrac{1}{2\rho}\cdot 
			\textup{div}\,(\nabla\rho)\cdot d\textup{vol}_{g}\nonumber\\
		&=&\int_{M}\,\bigg[\dfrac{1}{4}\,\dfrac{\|\nabla\rho\|^{2}}{\rho^{2}}
			-\dfrac{1}{2}\,\dfrac{\Delta \rho}{\rho}\bigg]
			\,\textup{div}\,(\rho\cdot \nabla\psi_{1})\cdot
			d\textup{vol}_{g}=-\int_{M}\,\dfrac{\Delta\big(\sqrt{\rho}\big)}{\sqrt{\rho}}\,\textup{div}\,
			(\rho\cdot\nabla\psi_{1})\cdot d\textup{vol}_{g}\,.
	\end{eqnarray}
	In the above computation, we have used the following formula, 	
	\begin{eqnarray}
		\dfrac{1}{4}\,\dfrac{\|\nabla u\|^{2}}{u^{2}}
			-\dfrac{1}{2}\,\dfrac{\Delta u}{u}=-\dfrac{\Delta\big(\sqrt{u}\big)}{\sqrt{u}}\,,
	\end{eqnarray}
	which is valid for every smooth function $u\,:\,M\rightarrow \mathbb{R}\,,$ as one may see after 
	a little computation. \\
	Now, \eqref{equation calcul des gradients symplectiques!},
	\eqref{equation immagration bureau} and Proposition \ref{proposition omega} yield 
	\begin{eqnarray}
		&&(d\mathcal{H})_{(\rho,\nabla\phi)}X_{(\rho,\nabla\phi)}=\nonumber\\
		&&\int_{M}\,\bigg[\dfrac{1}{2}\,\|\nabla\phi\|^{2}+V-\dfrac{\hslash^{2}}{2}\,
			\dfrac{\Delta\big(\sqrt{\rho}\big)}{\sqrt{\rho}}\bigg]\,\textup{div}\,
			(\rho\cdot \nabla\psi_{1})\cdot d\textup{vol}_{g}
			+\int_{M}\,g(\nabla\phi,\nabla\psi_{2})\,\rho\cdot d\textup{vol}_{g}\nonumber\\
		&=&-\int_{M}\,g\bigg(\nabla\psi_{1},\nabla
			\bigg[\dfrac{1}{2}\,\|\nabla\phi\|^{2}+V-\dfrac{\hslash^{2}}{2}\,
			\dfrac{\Delta\big(\sqrt{\rho}\big)}{\sqrt{\rho}}\bigg]\bigg)\,\rho\cdot d\textup{vol}_{g}
			+\int_{M}\,g(\nabla\phi,\nabla\psi_{2})\,\rho\cdot d\textup{vol}_{g}\nonumber\\
		&=&(\Omega_{\mathcal{L}})_{(\rho,\nabla\phi)}(X_{\mathcal{H}},X)\,.
	\end{eqnarray}
	The lemma follows. 
\end{proof}
\begin{remark}
	We observe (as it was intended to), that the flow generated by the symplectic gradient 
	$X_{\mathcal{H}}\in\mathfrak{X}(T\mathcal{D})$ corresponds exactly to the solutions of the Euler-Lagrange 
	equations on $\mathcal{D}$ associated to the Lagrangian $\mathcal{L}\,:\,T\mathcal{D}\rightarrow 
	\mathbb{R}$ introduced 
	in \eqref{definition du Lagrangien simplifie}, i.e., it satisfies the system of equations 
	\eqref{equation les equations d'euler-lagrange, enfin!!!} (with $X\cong 0$). \\
	We thus have a rigorous symplectic formulation of the Schr\"{o}dinger equation via its 
	hydrodynamical formulation which agrees with the corresponding Lagrangian formulation given in 
	Corollary \ref{corollaire du vraie lagrangien comme on aime!}. 
\end{remark}
\begin{lemma}\label{lemma gradients sympl. observables}
	For $\rho\in \mathcal{D}\,,$ $\nabla\phi\in \nabla C^{\infty}(M)$ and $F\,:\,TM\rightarrow
	\mathbb{R}\,,$ we have :
	\begin{eqnarray}\label{equation j'ecoute queen}
		(X_{\widehat{F}})_{(\rho,\nabla\phi)}=\Big(\rho,\nabla\phi,\mathbb{P}_{\rho}
		\big(\rho\,(\pi_{*}^{TM}\circ X_{F}\circ \nabla\phi\big),
		\nabla\big(F(\nabla\phi)\big)\Big)\,.	
	\end{eqnarray}
\end{lemma}
\begin{proof}
	As for the proof of Lemma \ref{lemme legradient symp de l'hami}, we will use 
	the vector field $X\in \mathfrak{X}(T\mathcal{D})$ introduced in 
	\eqref{equation definition X Y}, its flow $\varphi_{t}^{X}$ which is
	given in Lemma \ref{lemma le flowwww}, and the curve $\rho_{t}$ defined in the 
	proof of Lemma \ref{lemme legradient symp de l'hami} (see \eqref{equation derivee de rhooooo}). \\
	We have :
	\begin{eqnarray}\label{equation digestion, quand tu nous tiens!}
		&&(d\widehat{F})_{(\rho,\nabla\phi)}X_{(\rho,\nabla\phi)}=\dfrac{d}{dt}\bigg\vert_{0}\,
			\widehat{F}(\rho_{t},\nabla\phi+t\nabla\psi_{2})\nonumber\\
		&=&\dfrac{d}{dt}\bigg\vert_{0}\,\int_{M}\,
			F(\nabla\phi+t\nabla\psi_{2})\,\rho_{t}\cdot d\textup{vol}_{g}\nonumber\\
		&=&\int_{M}\Big[\mathbb{F}F(\nabla\phi)(\nabla\psi_{2})\,\rho+F(\nabla\phi)\,
			\textup{div}\,(\rho\cdot\nabla\phi)\Big]\cdot d\textup{vol}_{g}\nonumber\\
		&=&\int_{M}\Big[\mathbb{F}F(\nabla\phi)(\nabla\psi_{2})-
			g\big(\nabla\psi_{1},\nabla\big(F(\nabla\phi)\big)\big)\Big]\,
			\rho\cdot d\textup{vol}_{g}\,.
	\end{eqnarray}
	We need to transform the term $\mathbb{F}F(\nabla\phi)(\nabla\psi_{2})$ into a scalar product; to this 
	end, we will use the following formula 
	\begin{eqnarray}\label{equation formule sympatique!.!.}
		\mathbb{F}F(u_{x})(v_{x})=g_{{x}}\big(\pi^{TM}_{*_{u_{x}}}(X_{F})_{u_{x}},v_{x}\big)\,,
	\end{eqnarray}
	which holds whenever $u_{x},v_{x}\in T_{x}M\,,$ and where $X_{F}$ is the symplectic 
	gradient of $F$ with respect 
	to the symplectic form $\omega$ on $TM$ canonically associated to the metric $g\,.$ This 
	formula may be seen as follows. Recall that the canonical symplectic form 
	$\omega$ may be written (see \cite{Lang} and Example \ref{example blabla}) :
	\begin{eqnarray}\label{equation rapelle canonical sf digestion}
		\omega_{u_{x}}(A_{u_{x}},B_{u_{x}})=g_{x}(\pi^{TM}_{*_{u_{x}}}A_{u_{x}},KB_{u_{x}})-
		g_{x}(\pi^{TM}_{*_{u_{x}}}B_{u_{x}},KA_{u_{x}})\,,
	\end{eqnarray}
	where $u_{x}\in T_{x}M\,,$ $A_{u_{x}},B_{u_{x}}\in T_{u_{x}}TM$ and where $K\,:\,T(TM)\rightarrow TM$
	is the connector associated to the Riemannian metric $g\,.$ With 
	\eqref{equation rapelle canonical sf digestion}, it is a simple matter to derive 
	\eqref{equation formule sympatique!.!.} :
	\begin{eqnarray}
		\mathbb{F}F(u_{x})(v_{x})&=&\dfrac{d}{dt}\bigg\vert_{0}\,F(u_{x}+tv_{x})=(dF)_{u_{x}}\dfrac{d}{dt}
			\bigg\vert_{0}\,(u_{x}+tv_{x})=\omega_{u_{x}}\Big((X_{F})_{u_{x}},\dfrac{d}{dt}
			\bigg\vert_{0}\,(u_{x}+tv_{x})\Big)\nonumber\\
		&=&g_{x}\Big(\pi^{TM}_{*_{u_{x}}}(X_{F})_{u_{x}},K\dfrac{d}{dt}
			\bigg\vert_{0}\,(u_{x}+tv_{x})\Big)-g_{x}\Big(\pi^{TM}_{*_{u_{x}}}\dfrac{d}{dt}
			\bigg\vert_{0}\,(u_{x}+tv_{x}),K(X_{F})_{u_{x}}\Big)\nonumber\\
		&=&g_{x}\big(\pi^{TM}_{*_{u_{x}}}(X_{F})_{u_{x}},v_{x}\big)\,.
	\end{eqnarray}
	Of course, in the above computation we have used the following simple formulas:
	\begin{eqnarray}
		K\dfrac{d}{dt}\bigg\vert_{0}\,(u_{x}+tv_{x})=v_{x}\,\,\,\,\,\,\textup{and}\,\,\,\,\,\,
		\pi^{TM}_{*_{u_{x}}}\dfrac{d}{dt}\bigg\vert_{0}\,(u_{x}+tv_{x})=0\,.
	\end{eqnarray}
	Taking into account \eqref{equation formule sympatique!.!.}, we may rewrite 
	\eqref{equation digestion, quand tu nous tiens!} as
	\begin{eqnarray}
		&&\int_{M}\Big[\mathbb{F}F(\nabla\phi)(\nabla\psi_{2})-
			g\big(\nabla\psi_{1},\nabla\big(F(\nabla\phi)\big)\big)\Big]\,
			\rho\cdot d\textup{vol}_{g}\nonumber\\
		&=&\int_{M}\,\Big[g\big(\pi^{TM}_{*}\circ X_{F}\circ \nabla\phi,
			\nabla\psi_{2}\big)-
			g\big(\nabla\psi_{1},\nabla\big(F(\nabla\phi)\big)\big)\Big]\,
			\rho\cdot d\textup{vol}_{g}\nonumber\\
		&=&\int_{M}\Big[g\Big(\mathbb{P}_{\rho}\big(\rho\,(\pi^{TM}_{*}{\circ} X_{F}{\circ} \nabla\phi)\big),
			\nabla\psi_{2}\Big){-}
			g\big(\nabla\psi_{1},\nabla\big(F(\nabla\phi)\big)\big)\Big]
			\rho \cdot d\textup{vol}_{g}\,,\,\,\,\,\,\,\,\textbf{}
%		&=&(\Omega_{L})_{(\rho,\nabla\phi)}(X_{\widehat{F}},X)\,.
	\end{eqnarray}
	from which we see that $(d\widehat{F})X=
	\Omega_{\mathcal{L}}(X_{\widehat{F}},X)\,,$ with $X_{\widehat{F}}$ such as defined 
	in the right hand side of \eqref{equation j'ecoute queen}. The vector field $X_{\widehat{F}}$
	is thus the symplectic gradient of $F$ with respect to the symplectic form $\Omega_{\mathcal{L}}\,.$ 
	The lemma follows. 
\end{proof}
\begin{lemma}\label{lemme le crochet de poissonnnn}
	For $F,G\,:\,TM\rightarrow \mathbb{R}\,,$ we have :
	\begin{eqnarray}
		\{\widehat{F},\widehat{G}\}_{\mathcal{L}}=-\widehat{\{F,G\}_{L}}\,.
	\end{eqnarray}
\end{lemma}
\begin{proof}
	For $\rho\in\mathcal{D}\,,$ $\nabla\phi\in \nabla C^{\infty}(M)\,,$ and, in 
	view of Lemma \ref{lemma gradients sympl. observables}, we have :
	\begin{eqnarray}\label{equation pouette!! et repoutte}
		&&\{\widehat{F},\widehat{G}\}_{\mathcal{L}}(\rho,\nabla\phi)=
			(\Omega_{\mathcal{L}})_{(\rho,\nabla\phi)}(X_{\widehat{F}},
			X_{\widehat{G}})\nonumber\\
		&=&\int_{M}\,g\Big(\mathbb{P}_{\rho}\big(\rho\,(\pi^{TM}_{*}\circ X_{F}\circ \nabla\phi)
			\big),\nabla\big(G(\nabla\phi)\big)\Big)\,\rho\cdot d\textup{vol}_{g}\nonumber\\
		&&-\int_{M}\,g\Big(\mathbb{P}_{\rho}\big(\rho\,(\pi^{TM}_{*}\circ X_{G}\circ \nabla\phi)
			\big),\nabla\big(F(\nabla\phi)\big)\Big)\,\rho\cdot d\textup{vol}_{g}\nonumber\\
		&=&\int_{M}\,g\Big(\pi^{TM}_{*}\circ X_{F}\circ \nabla\phi,
			\nabla\big(G(\nabla\phi)\big)\Big)\,\rho\cdot d\textup{vol}_{g}\nonumber\\
		&&-\int_{M}\,g\Big(\pi^{TM}_{*}\circ X_{G}\circ \nabla\phi,
			\nabla\big(F(\nabla\phi)\big)\Big)\,\rho\cdot d\textup{vol}_{g}\,.
	\end{eqnarray}
	Moreover, we observe that if $X$ is a vector field on $M\,,$ then 
	\begin{eqnarray}
		&&g\big(\nabla\big(G(\nabla\phi)\big),X\big)=G_{*}(\nabla\phi)_{*}X
			=\omega\big(X_{G}\circ \nabla\phi,(\nabla\phi)_{*}X\big)\nonumber\\
		&=&g\big(\pi^{TM}_{*}\circ X_{G}\circ \nabla\phi,K(\nabla\phi)_{*}X\big)-
			g\big(\pi^{TM}_{*}\circ (\nabla\phi)_{*}X,KX_{G}\circ \nabla\phi\big)\nonumber\\
		&=&g\big(\pi^{TM}_{*}\circ X_{G}\circ \nabla\phi,\nabla_{X}\nabla\phi\big)
			-g\big(X,K X_{G}\circ \nabla\phi\big)\,,
	\end{eqnarray}
	and thus, denoting $\widetilde{X}_{F}:=\pi^{TM}_{*}\circ X_{G}\circ \nabla\phi$ and 
	$\widetilde{X}_{G}:=\pi^{TM}_{*}\circ X_{F}\circ \nabla\phi$ for simplicity, we may rewrite 
	\eqref{equation pouette!! et repoutte} as :
	\begin{eqnarray}\label{equation je commence a etre fatiqueeeeeee}
		&&\{\widehat{F},\widehat{G}\}_{\mathcal{L}}(\rho,\nabla\phi)=\nonumber\\
		&&\int_{M}\,g\big(\widetilde{X}_{G},\nabla_{\widetilde{X}_{F}}\nabla\phi\big)\,
			\rho\cdot d\textup{vol}_{g}
			-\int_{M}\,g\big(\widetilde{X}_{F},K X_{G}\circ\nabla\phi\big)\,\rho\cdot 
			d\textup{vol}_{g}\nonumber\\
		&&-\int_{M}\,g\big(\widetilde{X}_{F},\nabla_{\widetilde{X}_{G}}\nabla\phi\big)\,
			\rho\cdot d\textup{vol}_{g}
			-\int_{M}\,g\big(\widetilde{X}_{G},K X_{F}\circ\nabla\phi\big)\,\rho\cdot 
			d\textup{vol}_{g}\nonumber\\
		&=&-\bigg[\int_{M}\,g(\widetilde{X}_{F},K X_{G}\circ\nabla\phi)\,\rho\cdot d\textup{vol}_{g}
			-\int_{M}\,g(\widetilde{X}_{G},K X_{G}\circ\nabla\phi)\,\rho\cdot d\textup{vol}_{g}\bigg]
			\nonumber\\
		&&+\int_{M}\,g\big(\widetilde{X}_{G},\nabla_{\widetilde{X}_{F}}\nabla\phi\big)\,
			\rho\cdot d\textup{vol}_{g}-\int_{M}\,g\big(\widetilde{X}_{F},
			\nabla_{\widetilde{X}_{G}}\nabla\phi\big)\,
			\rho\cdot d\textup{vol}_{g}\nonumber\\
		&=&-\widehat{\{ F,G \}_{L}}(\rho,\nabla\phi)
			+\int_{M}\,g\big(\widetilde{X}_{G},\nabla_{\widetilde{X}_{F}}\nabla\phi\big)\,
			\rho\cdot d\textup{vol}_{g}-\int_{M}\,g\big(\widetilde{X}_{F},
			\nabla_{\widetilde{X}_{G}}\nabla\phi\big)\,
			\rho\cdot d\textup{vol}_{g}\,.
	\end{eqnarray}
	Clearly, we have to show that the last line in \eqref{equation je commence a etre fatiqueeeeeee} 
	vanishes. But this can be done easily with the help of the following formula 
	\begin{eqnarray}\label{equation I'm calling you!!!}
		g\big(X,\nabla_{Y}Z\big)-g\big(Y,\nabla_{X}Z\big)=-d(Z^{\sharp})(X,Y)\,,
	\end{eqnarray}
	which holds for every vector fields $X,Y,Z\in \mathfrak{X}(M)\,,$ and where $Z^{\sharp}$
	is the 1-form on $M$ defined by $(Z^{\sharp})_{x}(u_{x}):=g_{x}(Z_{x},u_{x})\,,$ $u_{x}\in T_{x}M\,.$ \\
	Using \eqref{equation I'm calling you!!!} and the fact that $d(d\phi)=0\,,$ one easily sees that 
	the last line in \eqref{equation je commence a etre fatiqueeeeeee} vanishes. The lemma follows. 
\end{proof}

\section{The almost Hermitian structure of $T\mathcal{D}$}\label{sss avant dernier}

	In \S\ref{section Euler-Lagrange etc.} 
	and \S\ref{section Hamiltonian}, we used 
	the usual techniques of geometric mechanics to find a Lagrangian and Hamiltonian description
	of the Schr\"{o}dinger equation, and we eventually arrived at the symplectic form $\Omega_{\mathcal{L}}$ 
	on $T\mathcal{D}$ which encodes the dynamics of a quantum particle and whose explicit description is given 
	in Proposition \ref{proposition omega}. \\
	In this section, we follow some ideas of \cite{Molitor-exponential} and show 
	that $\Omega_{\mathcal{L}}$ is the fundamental $2$-form of an almost Hermitian structure 
	on $T\mathcal{D}$ which comes from Dombrowski's construction \cite{Dombrowski} applied to a metric 
	$g^{\mathcal{D}}$ and a (non-metric) connection 
	$\nabla^{\mathcal{D}}$ on $\mathcal{D}\,,$ and discuss the integrability of this almost Hermitian structure.\\ 

	Let us start by recalling Dombrowki's construction. If $M$ is a manifold endowed with an affine 
	connection $\nabla\,,$ then Dombrowski splitting Theorem holds (see \cite{Dombrowski,Lang}) :
	\begin{eqnarray}
		T(TM)\cong TM\oplus TM\oplus TM\,,
	\end{eqnarray}
	this splitting being viewed as an isomorphism of vector bundles over $M\,,$ and the 
	isomorphism, say $\Phi\,,$
	being
	\begin{eqnarray}\label{equation Dombrowski}
	T_{u_{x}}TM\ni A_{u_{x}}\overset{\Phi}{\longmapsto} 
	\big(u_{x},(\pi^{M})_{*_{u_{x}}}A_{u_{x}},K^{M} A_{u_{x}}\big)\,,
	\end{eqnarray}
	where $\pi^{M}\,:\,TM\rightarrow M$ is the canonical projection and where 
	$K^{M}\,:\,T(TM)\rightarrow TM$ is the canonical connector associated to the connection 
	$\nabla^{}$ (see \cite{Lang}). 

	Having $A_{u_{x}}=\Phi^{-1}\big((u_{x},v_{x},w_{x})\big)\in T_{u_{x}}TM\,,$ we shall write, for 
	simplicity, $A_{u_{x}}=(u_{x},v_{x},w_{x})$ instead of $\Phi^{-1}\big((u_{x},v_{x},w_{x})\big)\,,$
	i.e., we will drop $\Phi\,.$ The second component 
	$v_{x}$ is usually referred to as the horizontal component of $A_{u_{x}}$ (with respect 
	to the connection $\nabla^{}$) and $w_{x}$ the vertical component.  \\

	With the above notation, and provided that $M$ is endowed with a Riemannian metric $g\,,$ 
	it is a simple matter to define on 
	$TM$ an almost Hermitian structure. 
	Indeed, we define a metric $g^{TM}\,,$ a 2-form $\omega^{TM}$
	and an almost complex structure $J^{TM}$ by setting
	\begin{eqnarray}\label{equation definition G, omega, etc.}
		g^{TM}_{u_{x}}\big(\big(u_{x},v_{x},w_{x}\big),
			\big({u}_{x},\overline{v}_{x},
			\overline{w}_{x}\big)\big)&:=&
			g^{}_{x}\big(v_{x},\overline{v}_{x}\big)+
			g^{}_{x}\big(w_{x},\overline{w}_{x}\big)\,,\nonumber\\
		\omega^{TM}_{u_{x}}\big(\big(u_{x},v_{x},w_{x}\big),
			\big({u}_{x},\overline{v}_{x},
			\overline{w}_{x}\big)\big)&:=&g^{}_{x}\big(v_{x},\overline{w}_{x}\big)-
			g^{}_{x}\big(w_{x},\overline{v}_{x}\big)\,,\nonumber\\
		J^{TM}_{u_{x}}\big(\big(u_{x},v_{x},w_{x}\big)\big)&:=&
		\big(u_{x},-w_{x},v_{x}\big)\,,
	\end{eqnarray}
	where $u_{x},v_{x},w_{x},\overline{v}_{x},\overline{w}_{x}
	\in T_{x}M\,.$\\
	Clearly, $(J^{TM})^{2}=-\textup{Id}$  and $g^{TM}(J^{TM}\,.\,,J^{TM}\,.\,)=g^{TM}(\,.\,,\,.\,)\,,$ which means that 
	$(TM,g^{TM},J^{TM})$ is an almost Hermitian manifold, and one readily sees that 
	$g^{TM},J^{TM}$ and $\omega^{TM}$ are compatible, i.e., that $
	\omega^{TM}=g^{TM}\big(J^{TM}\,.\,,\,.\,\big)\,;$ the $2$-form $\omega^{TM}$ is thus the fundamental 2-form of 
	the almost Hermitian manifold $(TM,g^{TM},J^{TM})\,.$ This is Dombrowski's construction.\\
\begin{example}\label{example blabla}
	Let $(M,g)$ be a (finite dimensional) Riemannian manifold with Levi-Civita connection $\nabla^{}\,,$ and let 
	$\omega=-d\theta$ be the canonical symplectic form\footnote{Recall that the canonical $1$-form $\theta$ on $T^{*}M$ 
	is defined, for $\alpha_{x}\in T_{x}^{*}M$
	and $A_{\alpha_{x}}\in T_{\alpha_{x}}T^{*}M\,,$ by 
	$\theta_{\alpha_{x}}(A_{\alpha_{x}}):=\alpha_{x}((\pi^{T^{*}M})_{*_{\alpha_{x}}}A_{\alpha_{x}})\,,$ where 
	$\pi^{T^{*}M}\,:\,T^{*}M\rightarrow M$ is the canonical projection.} 
	on $T^{*}M\,.$ Then the $2$-form $\omega^{TM}$ on $TM$ associated to $(g,\nabla^{})$ via Dombrowski's construction is 
	equal to the pull back of the canonical symplectic form $\omega$ via the Legendre transform $TM\rightarrow T^{*}M\,,
	v_{x}\mapsto g_{x}(v_{x},\,.\,)$ (see \cite{Lang}).
\end{example}
	In the case of the infinite dimensional manifold $\mathcal{D}\,,$ we already defined in \eqref{equation definition la une forme} 
	a metric $g^{\mathcal{D}}$ on $\mathcal{D}\,:$ 
	\begin{eqnarray}
	(g^{\mathcal{D}})_{\rho}\big((\rho,\nabla\phi),(\rho,\nabla\phi')\big):=\int_{M}\,
		g(\nabla \phi,\nabla\phi')\,\rho\,\cdot d\textup{vol}_{g}\,,
	\end{eqnarray}
	where $\rho\in \mathcal{D}$ and where $\nabla\phi,\nabla\phi'\in \nabla C^{\infty}(M)\,.$ We also used the following 
	identification (see \eqref{identification espaces densites}) :
	\begin{eqnarray}
		T(T\mathcal{D})\cong \mathcal{D}\times \nabla C^{\infty}(M)
			\times\nabla C^{\infty}(M)\times\nabla C^{\infty}(M)\,.
	\end{eqnarray}
	Clearly, this identification defines an affine connection $\nabla^{\mathcal{D}}$ on $\mathcal{D}$ whose 
	associated connector $K^{\mathcal{D}}$ is 
	\begin{eqnarray}
			K^{\mathcal{D}}\,:\,T(T\mathcal{D})\rightarrow T\mathcal{D}\,,\,\,\,
			(\rho,\nabla\phi,\nabla\psi_{1},\nabla\psi_{2})\mapsto (\rho,\nabla\psi_{2})
	\end{eqnarray}
	(one easily verifies that the above map has the properties of a connector).

	We thus have a triple $(\mathcal{D},g^{\mathcal{D}},\nabla^{\mathcal{D}})$ which yields, via Dombrowski's construction, 
	an almost Hermitian structure $(g^{T\mathcal{D}},J^{T\mathcal{D}},\omega^{T\mathcal{D}})$ on $T\mathcal{D}\,.$ For example, 
	\begin{eqnarray}
		&&(g^{T\mathcal{D}})_{(\rho,\nabla\phi)}\big((\rho,\nabla\phi,\nabla\psi_{1},\nabla\psi_{2}),
			(\rho,\nabla\phi,\nabla\alpha_{1},\nabla\alpha_{2})\big)\nonumber\\
		&=&\int_{M}\,g(\nabla\psi_{1},\nabla\alpha_{1})\,\rho\cdot d\textup{vol}_{g}+
			\int_{M}\,g(\nabla\psi_{2},\nabla\alpha_{2})\,\rho\cdot d\textup{vol}_{g}\,. 
	\end{eqnarray} 
	In particular, Proposition \ref{proposition omega} immediately yields  
\begin{proposition}
	The fundamental $2$-form $\omega^{T\mathcal{D}}$ of the almost Hermitian structure of $T\mathcal{D}$ associated to 
	$(g^{\mathcal{D}},\nabla^{\mathcal{D}})$ via Dombrowski's construction is $\Omega_{\mathcal{L}}\,,$ i.e. 
	\begin{eqnarray}
		\omega^{T\mathcal{D}}=\Omega_{\mathcal{L}}\,,
	\end{eqnarray}
	where $\Omega_{\mathcal{L}}=-d\Theta_{\mathcal{L}}$ has been defined in \eqref{equation je me suis gourré!!}\,.
\end{proposition}
\begin{remark}
	As we saw in \S\ref{section Hamiltonian}, the flow generated by the Hamiltonian vector field 
	$X_{\mathcal{H}}\in \mathfrak{X}(T\mathcal{D})$ 
	with respect to the symplectic form $\Omega_{\mathcal{L}}$ gives the dynamics of a quantum particle 
	under the influence of a potential $V$ (see \eqref{equation definition du l'hamiltonien} 
	for the definition of $\mathcal{H}\,:\,T\mathcal{D}\rightarrow \mathbb{R}$).
	Hence, and since $\Omega_{\mathcal{L}}=\Omega^{T\mathcal{D}}\,,$ we deduce that 
	the dynamics of a quantum particle is encoded in $(\mathcal{D},g^{\mathcal{D}},\nabla^{\mathcal{D}})\,.$
	This is analogous to the fact that the dynamics of a finite dimensional quantum system is encoded in the triple 
	$(\mathcal{P}_{n}^{\times},h_{F},\nabla^{(e)})\,,$ where $h_{F}$ and $\nabla^{(e)}$ are respectively the 
	Fisher metric and the exponential connection on $\mathcal{P}_{n}^{\times}$ 
	(see \cite{Molitor-exponential,Molitor-quantique}). In this sense, $g^{\mathcal{D}}$ and $\nabla^{\mathcal{D}}$ are 
	infinite dimensional analogues of $h_{F}$ and $\nabla^{(e)}\,.$
\end{remark}
	Let $T^{\mathcal{D}}$ and $R^{\mathcal{D}}$ be the torsion and the curvature tensor associated to the connection 
	$\nabla^{\mathcal{D}}\,,$ i.e., 
	\begin{description}
		\item[$\bullet$] $T^{\mathcal{D}}(X,Y)=
			\nabla^{\mathcal{D}}_{X}Y-\nabla^{\mathcal{D}}_{Y}X-[X,Y]\,,$ 
		\item[$\bullet$] $R^{\mathcal{D}}(X,Y)(Z)=\nabla^{\mathcal{D}}_{X}
			\nabla^{\mathcal{D}}_{Y}Z-\nabla^{\mathcal{D}}_{Y}
			\nabla^{\mathcal{D}}_{X}Z-\nabla^{\mathcal{D}}_{[X,Y]}Z\,,$ 
	\end{description}
	where $X,Y,Z\in \mathfrak{X}(\mathcal{D})\,.$ 

	By inspection of the proof of Lemma \ref{lemma le lie bracket!!!}, one easily finds that 
\begin{lemma}
	We have:	
	\begin{enumerate}
		\item $T^{\mathcal{D}}\big((\rho,\nabla\phi),(\rho,\nabla\psi)\big)=
			\Big(\rho,\mathbb{P}_{\rho}\big(\rho\,[\nabla\phi,\nabla\psi]
			\big)\Big)\,,$
		\item $R^{\mathcal{D}}\equiv 0\,,$
	\end{enumerate}
	where $\rho\in \mathcal{D}$ and $\nabla\phi,\nabla\psi\in \nabla C^{\infty}(M)\,,$ and where the operator $\mathbb{P}_{\rho}$ 
	has been defined in \eqref{equation HH decomposition}. 
	In particular, $\nabla^{\mathcal{D}}$ is not the Levi-Civita connection associated to $g^{\mathcal{D}}$ (its torsion is not trivial).
\end{lemma}
	Let $N^{T\mathcal{D}}$ be the Nijenhuis tensor of $J^{T\mathcal{D}}\,,$ i.e.,
	\begin{eqnarray}
		N^{T\mathcal{D}}(X,Y):=[X,Y]-[J^{TM}X,J^{TM}Y]+J^{T\mathcal{D}}[J^{T\mathcal{D}}X,Y]+
		J^{T\mathcal{D}}[X,J^{T\mathcal{D}}Y]\,,
	\end{eqnarray} 
	where $X,Y\in \mathfrak{X}(T\mathcal{D})\,.$

	Again, by inspection of the proof of Lemma \ref{lemma le lie bracket!!!}, one easily finds that
\begin{proposition}
	Let $J^{T\mathcal{D}}$ be the almost complex structure on $T\mathcal{D}$ associated to $(g^{\mathcal{D}},\nabla^{\mathcal{D}})$
	via Dombrowski's construction, and let $N^{T\mathcal{D}}$ be its Nijenhuis tensor. Then, 
	\begin{eqnarray}
		&&N^{T\mathcal{D}}\big((\rho,\nabla\phi,\nabla\psi_{1},\nabla\psi_{2}),(\rho,\nabla\phi,\nabla\alpha_{1},\nabla\alpha_{2})\big)=
			\nonumber\\
		&&\Big(\rho,\,\nabla\phi,\,\mathbb{P}_{\rho}\Big\{\rho[\nabla\alpha_{1},\nabla\psi_{1}]
			-\rho[\nabla\alpha_{2},\nabla\psi_{2}]\Big\},\,0\Big)
			+\Big(\rho,\,\nabla\phi,\,0,\,\mathbb{P}_{\rho}\Big\{\rho[\nabla\psi_{2},\nabla\alpha_{1}]
			+\rho[\nabla\psi_{1},\nabla\alpha_{2}]\Big\}\Big)\,,\,\,\,\,\,\,\text{}
	\end{eqnarray}
	where $\rho\in \mathcal{D}$ and where $\nabla\phi,\nabla\psi_{1},\nabla\psi_{2},\nabla\alpha_{1},\nabla\alpha_{2}\in \nabla C^{\infty}(M)\,.$
\end{proposition}
\begin{corollary}
	The almost Hermitian structure $J^{T\mathcal{D}}$ of $T\mathcal{D}$ is not integrable, i.e., $N^{T\mathcal{D}}\not\equiv 0\,.$
\end{corollary}
\section{Discussion : the wave function of a statistical manifold}\label{sss la fin}
	In \S\ref{section Euler-Lagrange etc.}, we associated to a time-dependant probability density function $\rho$ on a Riemannian 
	manifold $(M,g)$ a ``wave function" $\psi:=\sqrt{\rho}\,e^{-\frac{i}{\hbar}\phi}$ whose phase $\phi$ 
	is determined by solving the partial differential equation $\dot{\rho}=\textup{div}(\rho\nabla\phi)\,.$ 
	As we saw, this wave function linearizes the system of equations given in Proposition \ref{proposition les Euler-Lagrange} and 
	yields the usual Schr\"{o}dinger equation.
	
	In this section, \textit{which is mainly heuristic}, we discuss further the correspondence $\dot{\rho}\rightarrow \psi$ 
	through an example\footnote{This example has already been discussed in \cite{Molitor-exponential}, but 
	without any mathematical justifications.}, and make several comments and observations which 
	relate $\psi$ to representation theory, K\"{a}hler geometry, the geometrical formulation of quantum mechanics and quantization.\\

	Let us start with a simple example. 
	Let $\mathcal{N}(\mu,1)$ be the space of probability density functions $p(\xi;\mu)$ defined over $\mathbb{R}$ by  
	\begin{eqnarray}
		p(\xi;\mu):=\dfrac{1}{\sqrt{2\pi}}\,\textup{exp}\,\Big\{-\dfrac{(\mu-\xi)^{2}}2{}\Big\}\,,
	\end{eqnarray}
	where $\xi,\mu\in \mathbb{R}\,.$ 

	The set $\mathcal{N}(\mu,1)$ is a $1$-dimensional statistical manifold parameterized by the mean $\mu\in \mathbb{R}\,,$ i.e. 
	$\mathcal{N}(\mu,1)\cong \mathbb{R}\,.$ As one may easily 
	show (see \cite{Amari-Nagaoka}), the Fisher metric $h_{F}(\mu)$ is the Euclidean metric, and the exponential 
	connection $\nabla^{(e)}$ and the mixture connection $\nabla^{(m)}$ are equal to the canonical flat connection.
	Consequently (see \cite{Molitor-exponential}), $T\mathcal{N}(\mu,1)$ is naturally a K\"{a}hler manifold (via Dombrowski's construction)
	and one sees that $T\mathcal{N}(\mu,1)\cong\mathbb{C}$ via the map $b\,\partial_{\mu}\vert_{a}\mapsto a+i\,b\,.$
	
	Now, one of the most important ingredients of the geometrical formulation of quantum mechanics is the notion of \textit{K\"{a}hler functions}. 
	By definition, a smooth function $f\,:\,N\rightarrow \mathbb{R}$ on a K\"{a}hler manifold $N$ with K\"{a}hler structure $(g,J,\omega)$ 
	is a K\"{a}hler function if it satisfies $\mathcal{L}_{X_{f}}g=0\,,$ where $X_{f}$ is the 
	Hamiltonian vector field associated to $f\,,$ i.e. $\omega(X_{f},\,.\,)=df(.)\,,$ and where $\mathcal{L}_{X_{f}}$ is the Lie derivative 
	in the direction $X_{f}\,.$ 

	The space of K\"{a}hler functions $\mathscr{K}(N)$ on a K\"{a}hler manifold is always a finite dimensional Lie algebra for the 
	natural Poisson bracket $\{f,g\}:=\omega(X_{f},X_{g})\,.$ For example, when $N=\mathbb{P}(\mathbb{C}^{n})$ is the complex projective space, 
	then $\mathscr{K}(\mathbb{P}(\mathbb{C}^{n}))$ is isomorphic (in the Lie algebra sense) 
	to the space of $n\times n$ skew Hermitian matrices. Hence, K\"{a}hler functions are the natural geometric analogues of the 
	usual observables in quantum mechanics (see \cite{Ashtekar}).

	In the case $N=\mathbb{C}\,\,\,(\cong T\mathcal{N}(\mu,1))\,,$ it is not difficult to see that 
	the space $\mathscr{K}(\mathbb{C})$ of K\"{a}hler functions on $\mathbb{C}$ is spanned by 
	\begin{eqnarray}
		1,\,\,x,\,\,y,\,\,\dfrac{x^{2}+y^{2}}{2}
	\end{eqnarray}
	 (here $x$ and $y$ are respectively the real and imaginary parts of $z\in \mathbb{C}$), with the following 
	commutators 
	\begin{eqnarray} 
		\{1,\,.\,\}=0\,,\,\,\,\{x,y\}=1\,,\,\,\,\Big\{x,\dfrac{x^{2}+y^{2}}{2}\Big\}=y\,,\,\,\,
		\Big\{y,\dfrac{x^{2}+y^{2}}{2}\Big\}=-x\,.
	\end{eqnarray}

	The Lie algebra $\mathscr{K}(\mathbb{C})$ is related to quantum physics. 	
	If $p(t)$ is a smooth curve in $\mathcal{N}(\mu,1)\,,$ it is in particular a smooth curve in $\mathcal{D}(\mathbb{R})\,,$
	the space of smooth density probability functions\footnote{Even though $\mathbb{R}$ is not compact, 
	the space $\mathcal{D}(\mathbb{R})$ can be given 
	the structure of an infinite dimensional manifold, for example by using the convenient setting developed in \cite{Kriegl-Michor}.} 
	defined over $\mathbb{R}$ for the Lebesgue measure. Moreover, 
	if the time-derivative $\dot{p}(t)$ of $p(t)$ is identified with $x(t)+iy(t)\in \mathbb{C}\,,$ then a direct computation 
	shows that 
	\begin{eqnarray}\label{divergence macdo}
		\dfrac{d\,p(t)}{dt}=\textup{div}\,\big(p(t)\,\nabla \phi\big)\,,
	\end{eqnarray}
	where the (time-dependant) function $\phi\,:\,\mathbb{R}\rightarrow \mathbb{R}$ is defined (up to an additive constant) by 
	\begin{eqnarray}
		\phi(\xi)=y(t)\,\xi\,.
	\end{eqnarray}
	Hence, and taking into account \eqref{pas allemagne}, the derivative $\dot{p}(t)$ has an associated wave function 
	$\Psi\,:\, \mathbb{C}\rightarrow L^{2}(\mathbb{R},\mathbb{C})$ 
	which is defined, for $\xi\in \mathbb{R}$ and $z=x+iy\in \mathbb{C}\,,$ by 
	$\Psi(z)(\xi):=\sqrt{p(t)}\,e^{-\frac{i}{\hbar}\phi(\xi)}\,,$ i.e., 
	\begin{eqnarray}
		\Psi(z)(\xi):=\dfrac{1}{(2\pi)^{1/4}}\,\textup{exp}\bigg\{-\dfrac{(\xi-x)^{2}}{4}\bigg\}\,
		\exp\bigg\{-\dfrac{i}{\hbar}\,y\,\xi\bigg\}\,.
	\end{eqnarray}
	By construction, if $z=x+iy\,,$ then 
	\begin{eqnarray}
		\vert \Psi(z)(\xi)\vert^{2}=p(\xi;x)\,.
	\end{eqnarray}
	The map $\Psi$ is related to quantization and the geometrical formulation of quantum mechanics as follows.
	Let $\mathbf{Q}$ be the linear map from the space $\mathscr{K}(\mathbb{C})$ to the space of unbounded operators 
	acting on $L^{2}(\mathbb{R},\mathbb{C})$ which is defined by 
	\begin{eqnarray}
	1\mapsto Id,\,\,\,\,\,\,\,\,\,x\mapsto x\,,\,\,\,\,\,\,\,\,\,y\mapsto i\hbar\frac{\partial}{\partial x}\,,
	\,\,\,\,\,\,\,\,\frac{x^{2}+y^{2}}{2}\mapsto 
	-\frac{\hbar^{2}}{2}\frac{\partial^{2}}{\partial x^{2}}
	+\frac{1}{2}x^{2}-\Big(\frac{\hbar^{2}}{8}+\frac{1}{2}\Big)\,.
	\end{eqnarray}
	Observe that $\mathbf{Q}$ is ``essentially" the operator which quantizes the classical harmonic oscillator.
\begin{proposition}\label{envie de courir?}
	For all $f\in \mathscr{K}(\mathbb{C})$ and for all $z\in \mathbb{C}\,,$ we have :
	\begin{eqnarray}\label{eee bientot la fin j'espere}
		f(z)=\big\langle \Psi(z),\,\mathbf{Q}(f)\cdot \Psi(z)\big\rangle\,,
	\end{eqnarray}
	where $\langle\,,\,\rangle$ is the usual $L^{2}$-scalar product on $L^{2}(\mathbb{R},\mathbb{C})\,.$
\end{proposition}
\begin{proof}
	By direct calculations.
\end{proof}
	Because of the above proposition, we shall call $\Psi\,:\,\mathbb{C}\rightarrow L^{2}(\mathbb{R},\mathbb{C})$ the 
	\textit{wave function} associated to $\mathcal{N}(\mu,1)\,.$ Clearly, this wave function comes from the 
	embedding $\mathcal{N}(\mu,1)\subseteq \mathcal{D}(\mathbb{R})$ together with the fact that every element of $T\mathcal{D}$ 
	possesses a wave function (up to a phase, see \S\ref{section Euler-Lagrange etc.}). \\

%	One interesting feature of this wave function is that it ``quantizes" the points of $\mathbb{C}$ in the sense that to every 
%	$z\in \mathbb{C}$ is attached a ``genuine" wave function $\Psi(z)\,:\,\mathbb{R}\rightarrow \mathbb{C}\,,$ the latter being
%	clearly related, in this example, to the quantum harmonic oscillator. This contrasts 
%	with the usual canonical quantization where one focusses on observables and on the definition of an operator similar 
%	to $\mathbf{Q}\,.$ 
%
%	Also, we observe that the space of K\"{a}hler functions on $\mathbb{C}$ is completely characterized by $\Psi$ and the ``quantizing
%	operator" $\textbf{Q}\,,$ and thus, these K\"{a}hler functions have a quantum mechanical meaning. This shows in particular 
%	that the geometrical formulation of quantum mechanics, based on the K\"{a}hler properties of the complex projective space, 
%	can be extended to more general K\"{a}hler manifolds (here $\mathbb{C}$) while keeping a physical meaning.  \\

	More generally, if $S$ is a submanifold of the space $\mathcal{D}$ of probability density functions defined 
	on an oriented (compact and connected) Riemannian manifold $(M,g)\,,$ then $TS\subseteq T\mathcal{D}\,,$ and thus, to every time-dependant 
	probability density function $\rho\,,$ there is, by  solving the equation $\dot{\rho}=\textup{div}(\rho\nabla\phi)\,,$ 
	an associated wave function $\Psi=\sqrt{\rho}\,e^{-\frac{i}{\hbar}\phi}\,.$ We thus get a map 
	that we call the \textit{wave function} associated to $S$ (which is, strictly speaking, only defined up to a phase factor) : 
	\begin{eqnarray}
		\Psi\,:\,TS\rightarrow L^{2}(M,\mathbb{C})\,.
	\end{eqnarray}

	The above wave function is an infinite dimensional generalization of a wave function that we already 
	considered\footnote{In \cite{Molitor-exponential},
	we use a different notation.} 
	in \cite{Molitor-exponential}. In the latter paper, we consider a finite set $\Omega:=\{x_{1},...,x_{n}\}$ 
	on which we define the space $\mathcal{P}_{n}^{\times}$ of 
	positive probabilities $p$ on $\Omega\,,$ i.e. $p\,:\,\Omega\rightarrow \mathbb{R}\,,$ $p>0\,,$ $\sum_{k=1}^{n}p(x_{k})=1\,.$
	The space $\mathcal{P}_{n}^{\times}$ is a finite dimensional statistical manifold. If $z_{p}=dp(t)/dt\vert_{0}$ is a tangent 
	vector at $p\in\mathcal{P}_{n}^{\times}\,,$ then we construct a wave function $\Psi\,:\,T\mathcal{P}_{n}^{\times}\rightarrow
	L^{2}(\Omega,\mathbb{C})\cong \mathbb{C}^{n}$ as 
	follows:
	\begin{eqnarray}
		\Psi(z_{p})(x_{k}):=\sqrt{p(x_{k}})\,e^{iu_{k}/2}\,,
	\end{eqnarray} 
	where $u_{k}\in \mathbb{R}$ is defined, for $k=1,...,n\,,$ by 
	\begin{eqnarray}\label{macdo encore et encore}
		\dfrac{d\,p(t)(x_{k})}{dt}\,\bigg\vert_{0}=u_{k}p(x_{k})\,.
	\end{eqnarray}
	Equation \eqref{macdo encore et encore}	is a finite dimensional analogue of \eqref{divergence macdo}.

	Using the above ``finite dimensional" wave function, we were able in \cite{Molitor-exponential} to establish an analogue of 
	Proposition \ref{envie de courir?} in the case of the binomial distribution $\mathcal{B}(n,q)$ defined over $\{0,1,...,n\}\,,$
	the latter being viewed as a subspace of $\mathcal{P}_{n+1}^{\times}$ 
	(see \cite{Molitor-exponential}, Proposition 9.7 and Lemma 9.8), and to conclude
	that the spin of particle in a Stern-Gerlach experiment is encoded in $\mathcal{B}(n,q)\,.$\\

	These examples suggest that a ``moving probability density function" always possesses an associated wave function, and that 
	the latter, in good cases, is related to representation theory, quantization, and of course to 
	the natural almost Hermitian structure of the underlying 
	statistical manifold. More important, this suggests that the usual concepts of the standard quantum 
	formalism (wave functions, Hilbert spaces, Hermitian operators, etc.)
	may be mathematically derived from more primitive concepts, rooted in statistics and information geometry.

	To clarify these foundational aspects of quantum mechanics would be particularly interesting, especially in view of quantum gravity.

	\textbf{}\\
\textbf{Acknowledgements.} I would like to thank Yoshiaki Maeda, Hsiung Tze and Tilmann Wurzbacher for many helpful discussions.
	
	This work was done with the financial support of the Japan Society for the Promotion of Science.

%\begin{footnotesize}\bibliography{bibala}\end{footnotesize}

\end{document}